\documentclass{birkjour}

 \newtheorem{theorem}{Theorem}[section]
 \newtheorem{corollary}[theorem]{Corollary}
 \newtheorem{lemma}[theorem]{Lemma}
 \newtheorem{proposition}[theorem]{Proposition}
 \theoremstyle{definition}
 \newtheorem{definition}[theorem]{Definition}
 \theoremstyle{remark}
 \newtheorem{remark}[theorem]{Remark}
 \newtheorem*{example}{Example}
 \numberwithin{equation}{section}

\usepackage{amssymb,latexsym, amsmath, enumerate, mathrsfs, esint}
\usepackage{color}
\usepackage{amsthm}
\usepackage{graphicx}
\usepackage{subfigure}
\usepackage{tikz}
\usepackage[utf8]{inputenc}
\usepackage{pdfsync}
\usepackage{url}
\numberwithin{equation}{section}

\def\R{\mathbb R}
\def\L{\mathcal{L}}

\def\N{\mathbb N}
\def\C{\mathbb C}

\renewcommand{\epsilon}{\varepsilon}
\renewcommand{\phi}{\varphi}

\def\v{\upsilon}

\def\<{\langle} \def\>{\rangle}

\renewcommand{\thesubsection}{\thesection.\arabic{subsection}}

\begin{document}

\title [Commutators for certain fractional type operators]{Commutators for certain fractional type operators on weighted spaces and Orlicz-Morrey spaces}

\author[Huoxiong Wu]{Huoxiong Wu}

\address{
School of Mathematical Sciences, Xiamen University\\
Xiamen, Fujian, 361005\\
People’s Republic of China}

\email{huoxwu@xmu.edu.cn}

\author{Tong Zhang$^*$}

\address{
School of Mathematical Sciences, Fudan University\\
Shanghai, 200433\\
People’s Republic of China}
\email{tongzhang18@fudan.edu.cn}




\begin{abstract}
In this paper, we focus on a class of fractional type integral operators that can be served as extensions of Riesz potential with kernels
$$K(x,y)=\frac{\Omega_1(x-A_1 y)}{|x-A_1 y |^{\frac{n}{q_1}}} \cdots \frac{\Omega_m(x-A_m y)}{|x-A_m y |^{\frac{n}{q_m}}},$$
where $\alpha\in [0,n),  m\geqslant1, \sum_{i=1}^m\frac{n}{q_i}=n-\alpha$, $\{A_i\}^m_{i=1}$ are invertible matrixes, $\Omega_i$ is homogeneous of degree 0 on $\R^n$ and $\Omega_i\in L^{p_i}(S^{n-1})$ for some $p_i\in [1,\infty)$. Under appropriate assumptions, we obtain the weighted $L^p$ estimates as well as weighted Hardy estimates of the commutator for such operators with $BMO$-type function. In addition, we acquire the boundedness of these operators and their commutators with a function in Campanato space on Orcliz-Morrey spaces as well as the compactness for such commutators in a special case: $m=1$ and $A=I$.
\end{abstract}

\maketitle

\section{Introduction}
The boundedness of singular integral operators on various function spaces has always been a classic topic in Harmonic Analysis.  During the last years, there are several kinds of generalizations for classical operators such as the Calder\'on-Zygmund operator and the Riesz potential $I_\alpha$.

Originally, Ricci and Sj\"ogren considered the $L^2(\R)$-boundedness of the integral operator $T_1$ defined by
$$T_1f(x)=\int_{\R} |x-y|^{-\alpha}|x+y|^{\alpha-1}f(y)dy,$$
where $0<\alpha<1$ in \cite{RS88} when they focused on some maximal operators coming up in studying the boundary behavior of Possion integrals. Then Godoy and Urciuolo extended the operator to $\R^n$ in \cite{GU93} and obtained the $L^p(\R^n)$-boundedness of the operator $T_n$ defined as
$$T_nf(x)=\int_{\R^n} |x-y|^{-\alpha}|x+y|^{\alpha-n}f(y)dy
$$
for $1<p<\infty$. Since then, a significant number of related operators in more general conditions have been studied extensively, such as the operator $\tilde T$ defined by
$$\tilde Tf(x)=\int_{\R^n} \frac{1}{|x-A_1y|^{\alpha_1}\cdots|x-A_m y|^{\alpha_m}}f(y)dy,$$
where $\{A_i\}^m_{i=1}$ are invertible matrixes and $\alpha_1+\cdots+\alpha_m=n-\alpha$ for some $\alpha\in [0,n)$.  This operator is bounded from $L^p(\R^n)$ to $L^q(\R^n)$ for any $1<p<\frac{n}{\alpha}$ and  $\frac{1}{q}=\frac{1}{p}-\frac{\alpha}{n}$, from $H^p(\R^n)$ to $L^q(\R^n)$ for any $0<p<\frac{n}{\alpha}$ and  $\frac{1}{q}=\frac{1}{p}-\frac{\alpha}{n}$. Furthermore, there are several weighted estimates for $\tilde T$ in different  conditions, as well as some results on variable spaces. For the abundant results of boundedness for $\tilde T$, we refer the readers to  \cite{GU99, RU12,RU13, RU14(2), RU17(2)}.

In this work, we investigate a class of fractional type integral operators with rough kernels sharing the following form:
Given $\alpha\in[0,n)$, $m\geqslant1$, $  \sum_{i=1}^m\frac{n}{q_i}=n-\alpha$ and a suitable function $f$, $T_{\Omega,\alpha}^m f$ is defined by
\begin{align}\label{D1}
	T_{\Omega,\alpha}^m f(x)=\int_{\R^n}\frac{\Omega_1(x-A_1 y)}{|x-A_1 y |^{\frac{n}{q_1}}} \cdots \frac{\Omega_m(x-A_m y)}{|x-A_m y |^{\frac{n}{q_m}}} f(y)dy.
\end{align}
where $\Omega_i$ is homogeneous of degree 0 on $\R^n$ and $\Omega_i\in L^{p_i}(S^{n-1})$ for some $p_i\geqslant1$. If we set $m=1$, $\Omega_1=1$ and $A_1=I$, then $T_{\Omega,\alpha}^m$ is indeed equivalent to $I_\alpha$, the Riesz Potential of order $\alpha$. While for the special case $m=1$ and $A_1=I$, set
\begin{align}\label{D2}
	\tilde T_{\Omega,\alpha}^1 f(x)=\int_{\R^n}\frac{\Omega(x- y)}{|x- y|^{n-\alpha}} f(y)dy
\end{align}
for suitable function $f$ where $\Omega$ is homogeneous of degree 0 on $\R^n$ and $\Omega\in L^{r}(S^{n-1})$ for some $r\geqslant1$.
Under several assumptions, there have been a series of boundedness results in weighted spaces. In \cite{DL98, DL00}, Ding and Lu gave the  $H^p(\R^n)-L^q(\R^n)$ , $H^p(\R^n)-H^q(\R^n)$ and $L^p_{\omega^p}(\R^n)-L^q_{\omega^q}(\R^n)$ boundedness of $\tilde T_{\Omega,\alpha}^1$ when $q=\frac{np}{n-\alpha p}$.
Ding, Lee and Lin obtain the  boundedness of $\tilde T_{\Omega,\alpha}^1$ from $H^p_{\omega^p}(\R^n)$ to $L^q_{\omega^q}(\R^n)$ and from $H^p_{\omega^p}(\R^n)$ to $H^q_{\omega^q}(\R^n)$ when $\alpha>0$ in \cite{DLL03,LL02}. While for $\alpha=0$, weighted estimates for $\tilde T_{\Omega,0}^1$ have been investigated in \cite{D93,KW79,W90}.

Weighted estimates have already been contributed to the more general case $m>1$, where the kernel has more than one singularity and meets certain H\"omander condition. For example, Godoy and Urciuolo obtained the $L^p(\R^n)-L^{\frac{np}{n-\alpha p}}(\R^n)$ boundedness of $T_{\Omega,\alpha}^m$ in \cite{GU99}  under appropriate assumptions, Rocha and Urciuolo obtained the $H^p(\R^n)-L^{\frac{np}{n-\alpha p}}(\R^n)$ boundedness of $T_{\Omega,\alpha}^m$ in \cite{RU12}, Riveros and Urciuolo obtained the weighted $L^p-L^{\frac{np}{n-\alpha p}}$, weighted $BMO$ and weak-type estimates of $T_{\Omega,\alpha}^m$ in \cite{RU14} as well as part of the weighted results on weighted Hardy space for $T_{\Omega,\alpha}^m$ in \cite{RU17}. The behaviors of the commutators associated with certain fractional type operators in weighted spaces and variable Lebesgue spaces have been discussed in \cite{IR18,IV23} recently.

Recall the definitions of the $Orlicz-Morrey$ space and the generalized $Campanato$ space:
\begin{definition}
	Let $\Phi$ be a Young function and $\phi: (0,\infty)\to(0,\infty)$. For any ball $B=B(x,R)$, set
	$$\|f\|_{\Phi,\phi,B}:=\inf\bigg\{\lambda>0: \frac{1}{\phi(R)}\fint_B \Phi\bigg(\frac{|f(x)|}{\lambda}\bigg)\,dx\leqslant1\bigg\}$$
	and
	$$\|f\|_{L(\Phi, \phi)}=\sup\limits_{B}\|f\|_{\Phi,\phi,B},$$
	where the supremum is taken over all balls $B\subset \R^n$.
	The $Orliz-Morrey$ space $L(\Phi, \phi)(\R^n)$ consists of all functions $f$ with  $\|f\|_{L(\Phi, \phi)}<\infty$ and is a Banach space.
\end{definition}

\begin{definition}
	For $\psi:(0,\infty)\to(0,\infty)$ and $1\leqslant p<\infty$, set
	$$\|f\|_{\mathcal{L}_{p,\psi}}:=\sup\limits_{B=B(x,R)}\frac{1}{\psi(R)}\bigg(\fint_B |f(y)-\fint_ B f|^p\bigg)^{\frac{1}{p}},$$
	where the supremum is taken over all balls $B\subset \R^n$. The generalized $Campanato$ space  $\mathcal{L}_{p,\psi}(\R^n)$ consists of all functions $f$ with  $\|f\|_{\mathcal{L}_{p,\psi}}<\infty$.
\end{definition}

How the commutators of integral operators with functions in $Campanato$ spaces behaves in $Orlicz-Morrey$ spaces have undergone extensive research in recent years. For instance, Nakai, Arai, Yamaguchi and Shi focused on  the commutators $[b,T]$ and $[b,I_\rho]$ where $b$ is in generalized $Campanato$ spaces, $T$ is the Calder\'on-Zygmund operator and $I_\rho$ is a generalized fractional integral operator. They obtained the results of boundedness and compactness for those operators in \cite{SAN21, YN22} where they showed that $[b,T]$ and $[b,I_\rho]$ are bounded from $L(\Phi,\phi)(\R^n)$ to $L(\Psi,\phi)(\R^n)$ for any  $b\in \mathcal{L}_{1,\psi}(\R^n)$ and compact from $L(\Phi,\phi)(\R^n)$ to $L(\Psi,\phi)(\R^n)$ for any $b\in \overline{C_c^\infty(\R^n)}^{\|\cdot\|_{\mathcal{L}_{1,\psi}}}$ under corresponding restrictions on $\phi$, $\Phi$ and $\Psi$.

Nonetheless, the boundedness and compactness for the operators defined by (\ref{D1}) and their commutators on various spaces are far from complete in case of the challenge of dealing with matrixes $\{A_i\}^m_{i=1}$. To this end, we will take into consideration a new set denoted by $BMO^A(\R^n)$, which will be defined in Section 3 as well as a set denoted by $\mathcal{L}^A_{1,\psi}(\R^n)$ which will be introduced in
Section 4. To our knowledge, these types of functions have never been defined before.

Inspired by the works \cite{HW22,RU14,RU17,SAN21,YN22}, the aim of this paper is to investigate how $[b,T_{\Omega,\alpha}^m]$ behaves in weighted spaces and $Orcliz-Morrey$ spaces, with a focus on the boundedness and compactness. Our results are comprehensive and systematic, covering most of the conclusions that are already known and limited to the relevant spaces.

The following is the order in which this paper is organized. Section 2 provides the necessary preliminaries for presenting the definitions and results in the following sections. Our main results are completely represented in Sections 3, 4, 5 which states the following:
\begin{itemize}
	\item[$\bullet$] $[b,T_{\Omega,\alpha}^m]$ can be extended to a bounded operator from $L^p_{\omega^p}(\R^n)$ to $L^q_{\omega^q}(\R^n)$ for $b\in BMO(\R^n)$ and $\omega$ satisfying $\omega^s\in \mathcal{A}(\frac{p}{s}.\frac{q}{s})$ where $s$ is a constant such that $\frac{1}{p_1}+\cdots+\frac{1}{p_m}+\frac{1}{s}=1$;
	
	\item[$\bullet$]  $[b,T_{\Omega,\alpha}^m]$ can be extended to a bounded operator from $H^p_{\omega^p}(\R^n)$ to $L^q_{\omega^q}(\R^n)$ for any $\omega$ satisfying $\omega^{\frac{sn}{n-\alpha s}} \in \mathcal{A}_1 $ and $b\in \cap_i BMO^{A_i}(\R^n)\cap BMO_{\omega^p,p}(\R^n)$;
	\\\\Besides, under several restrictions on $\phi$, $\Phi$ and $\Psi$, we obtain:
	\item[$\bullet$] $T_{\Omega,\alpha}^m$ is bounded from $L(\Phi,\phi)(\R^n)$ to $L(\Psi,\phi)(\R^n)$;
	\item[$\bullet$]  $[b,T_{\Omega,\alpha}^m]$ is bounded from $L(\Phi,\phi)(\R^n)$ to $L(\Psi,\phi)(\R^n)$ for $b\in\cap_{i=1}^m\mathcal{L}^{A_i}_{1,\psi}(\R^n)$;
	
	\item[$\bullet$] $[b,\tilde T_{\Omega,\alpha}^1]$ is compact from $L(\Phi,\phi)(\R^n)$ to $L(\Psi,\phi)(\R^n)$ for $b\in \overline{C_c^\infty(\R^n)}^{\|\cdot\|_{\mathcal{L}_{1,\psi}}}$.
\end{itemize}


\section{Preliminary}
The necessary notations, definitions, and lemmas are listed in this section for the proof of our main theorems.

\subsection{Notations}
The symbol $A\lesssim B$ means $A\leqslant CB$ for some constant $C$ and we denote $A\sim B$ if $A\lesssim B$ and $B\lesssim A$ hold at the same time. $B(x,r)$ denotes the ball with center $x$ and radius $r$ while $CB(x,r):=B(x,Cr)$ for any $C>0$ and $|B|$ denotes the Lebesgue measure of $B$. We briefly denote $B(0,R)$ by $B(R)$ for any $R>0$. For a positive constant $a$, $\left\lfloor a\right\rfloor$ stands for the integer part of $a$. For any function $f$ and ball $B$,  we denote by $\fint_{B} f$ the average value of $f$ on $B$, i.e. $\frac{1}{|B|} \int_B f$.

\subsection{Muckenhoupt Weights}
A non-negative measurable function $\omega$ lies in $\mathcal{A}_p$ with $1<p<\infty$ if
$$\sup\limits_{B} \bigg(\frac{1}{|B|}\int_{B}\omega(x)\,dx\bigg)\bigg(\frac{1}{|B|}\int_{B}\omega^{1-p'}(x)\,dx\bigg)^{p-1}<\infty,$$
where the supremum is taken over all balls $B\subset \R^n$ and in $\mathcal{A}_1$ if there exists $C>0$ such that
$$\frac{1}{|B|}\int_{B}\omega(x)\,dx\leqslant C ess \inf\limits_{x\in B} \omega(x)$$
for any ball $B\subset \R^n$.
If there exists $r>1$ and $C>0$ such that
$$\bigg(\frac{1}{|B|}\int_{B}\omega^r(x)\,dx\bigg)^{\frac{1}{r}}\leqslant C \frac{1}{|B|}\int_{B}\omega(x)\,dx$$
for any ball $B\subset \R^n$, then $\omega$ is said to satisfy the reverse H\"older inequality of order $r$ and is denoted by $\omega\in RH_r$.
It is well-known that for any $1\leqslant t\leqslant s<\infty$, $\mathcal{A}_t\subset \mathcal{A}_s$ and $RH_s \subset RH_t$. Denote $\mathcal{A}_\infty:=\cup_{p\geqslant1} \mathcal{A}_p$, $q_\omega:=\inf\{q>1: \ \omega\in \mathcal{A}_q\}$ and $r_\omega:=\sup\{r>1: \omega\in RH_r\}$. Besides, if $\omega\in \mathcal{A}_p$ for some $1\leqslant p \leqslant \infty$, then $\omega^\alpha \in \mathcal{A}_p$ for any $0<\alpha<1$ and $\omega^\beta \in \mathcal{A}_p$ for some $\beta>1$.
\begin{lemma}\rm(\cite{SW85})\em
	If $r>1$, then $\omega^r \in  \mathcal{A}_\infty$ if and only if $\omega\in RH_r$.
\end{lemma}
\begin{lemma}\rm(\cite{J00,GW73})\em
	If $\omega\in \mathcal{A}_p\cap RH_r$ with $p\geqslant 1$ and $r>1$, then
	$$\bigg(\frac{|S|}{|B|}\bigg)^p\lesssim\frac{\omega(S)}{\omega(B)}\lesssim \bigg(\frac{|S|}{|B|}\bigg)^{\frac{r-1}{r}}$$
	for any measurable subset $S$ of a ball $B$.
\end{lemma}

\subsection{Young function, $\textbf{Orlicz-Morrey}$ space and generalized $\textbf{Campanato}$ space}

In this subsection, we briefly state some basic definitions and properties of certain classes of Young functions, $Orlicz$ space and other function spaces. For more details, see \cite{SAN21,YN22} and the references.
\begin{definition}
	(1) $\mathcal{G}^{inc}$ denotes the set of all functions $\phi:(0,\infty)\to(0,\infty)$ such that there exists $C>0$ satisfying
	$$\phi(t_1)\leqslant C\phi(t_2)\ \ \ and\ \ \ C\frac{\phi(t_1)}{t_1}\geqslant\frac{\phi(t_2)}{t_2}$$
	for any $0<t_1<t_2<\infty$.\\
	(2) $\mathcal{G}^{dec}$ denotes the set of all functions $\phi:(0,\infty)\to(0,\infty)$ such that there exists $C>0$ satisfying
	$$C\phi(t_1)\geqslant \phi(t_2)\ \ \ and\ \ \ \phi(t_1)t_1^n\leqslant C \phi(t_2)t_2^n$$
	for any $0<t_1<t_2<\infty$.
\end{definition}
\begin{remark}
	(1) If $\phi\in \mathcal{G}^{inc}\cup\mathcal{G}^{dec}$, then there exists $C>0$ such that
	$$\frac{1}{C}\phi(s)\leqslant \phi(t) \leqslant C\phi(s)\ \ \ \ if\ \ \ \ \frac{1}{2}s\leqslant t\leqslant 2s.$$
	(2)If $\phi\in \mathcal{G}^{dec}$ and
	there exists $C>0$ such that for any $r>0$
	\begin{align}\label{phi}
		\int_r^\infty \frac{\phi(t)}{t}\,dt\leqslant C \phi(r),
	\end{align}
	then $\lim\limits_{r\to \infty} \phi(r)=0\ \ and\ \ \lim\limits_{r\to 0} \phi(r)=\infty$.
\end{remark}

\begin{definition}
	(1) $\Phi: [0,\infty]\to[0,\infty]$ is said to be a Young function satisfying $\Delta_2$-condition, denoted by $\Phi\in\Delta_2$, if $\Phi$ is convex, increasing, left continuous,
	$$\lim\limits_{t\to 0^+} \Phi(t)=\Phi(0)=0,\ \ \lim\limits_{t\to \infty} \Phi(t)=\Phi(\infty)=\infty,$$
	and there exists $C>0$ such that $\Phi(2t)\leqslant C\Phi(t)$ for any $t>0$.\\
	(2) $\Phi: [0,\infty]\to[0,\infty]$ is said to be a Young function satisfying $\nabla_2$-condition, denoted by $\Phi\in\nabla_2$, if $\Phi$ is convex, increasing, left continuous,
	$$\lim\limits_{t\to 0^+} \Phi(t)=\Phi(0)=0,\ \ \lim\limits_{t\to \infty} \Phi(t)=\Phi(\infty)=\infty,$$
	and there exists $K>1$ such that $\Phi(t)\leqslant \frac{1}{2K}\Phi(Kt)$ for any $t>0$.
\end{definition}
\begin{remark}\label{R2}
	(1) If $\Phi\in  \Delta_2\cup\nabla_2$ and  $\Phi^{-1}$ denotes the inverse function of $\Phi$, then there exists $C>0$ such that
	$$\frac{1}{C}\Phi^{-1}(s)\leqslant \Phi^{-1}(t) \leqslant C\Phi^{-1}(s)\ \ \ \ if\ \ \ \ \frac{1}{2}s\leqslant t\leqslant 2s;$$
	(2) If $\Phi\in \nabla_2$, then there will exist $q>1$ and $C>0$ such that $\frac{\Phi(t)}{t^q}\leqslant C \frac{\Phi(s)}{s^q}$ for any $t<s$;\\	
	(3) If $\Phi\in \Delta_2$, then there will exist $q\geqslant 1$ and $C>0$ such that $\frac{\Phi(t)}{t^q}\geqslant C \frac{\Phi(s)}{s^q}$ for any $t<s$.
\end{remark}

Here and subsequently, $\Phi$ will stand for some Young function satisfying $\nabla_2$-condition or $\Delta_2$-condition. Let $\tilde\Phi\colon[0,\infty)\to[0,\infty)$ be the \emph{complementary   function} to $\Phi$, defined as
\begin{equation}\label{eq:complementary}
	\tilde\Phi(t) := \sup_{s\in(0,\infty)} \big\{st - \Phi(s)\big\}.
\end{equation}
This definition ensures that \emph{Young's inequality} holds:
\begin{equation}\label{eq:Youngsineq}
	s t \leqslant \Phi(s) + \tilde\Phi(t)\quad\text{for } s,t\in[0,\infty).
\end{equation}

Recall the definitions of the $Orlicz-Morrey$ space and the generalized $Campanato$ space defined before and then we proceed with the following definitions and related lemmas.
\begin{definition}$(Orlicz$ space$)$
	Set $$\|f\|_{L^\Phi}:=\inf\bigg\{\lambda>0:\int_{\R^n}
	\Phi(\frac{|f(x)|}{\lambda})\,dx\leqslant 1\bigg\}$$
	and the $Orlicz$ space $L^\Phi(\R^n)$ is the set of all measurable function $f$ with $\|f\|_{L^\Phi}<\infty$.
\end{definition}

\begin{lemma}\label{Pro}\rm(\cite{SAN21})\em
	(1) If $\Phi\in \Delta_2\cap\nabla_2$, $\phi\in \mathcal{G}^{dec}$, then
	$$\|\chi_{B(x,R)}\|_{L(\Phi,\phi)}\sim \frac{1}{\Phi^{-1}(\phi(R))};$$
	(2) If $\Phi\in \Delta_2\cup\nabla_2$,
	$$\fint_{B(x,R)} |f(y|\,dy\leqslant 2 \Phi^{-1}(\phi(R))\|f\|_{\Phi,\phi,B(x,R)};$$
	(3) If $\Phi\in \nabla_2$, then there exists $p>1$ such that
	$$\bigg(\fint_{B(x,R)}|f(y)|^p\,dy\bigg)^{\frac{1}{p}}\lesssim \Phi^{-1}(\phi(R))\|f\|_{\Phi,\phi,B(x,R)}$$
	and we denote the set of all $p$ satisfying the above condition by $p_{\nabla_2}(\Phi)$.
\end{lemma}

\begin{lemma}\rm(\cite[Corollary 4.3]{AN18})
	\em If $\psi\in \mathcal{G}^{inc}$ and $1<p<\infty$, then
	$$\mathcal{L}_{1,\psi}(\R^n)=\mathcal{L}_{p,\psi}(\R^n)\ \ \ and \ \ \ \|\cdot\|_{\mathcal{L}_{1,\psi}}\sim\|\cdot\|_{\mathcal{L}_{p,\psi}},$$
	where the supremum is taken over all balls $B\subset \R^n$.
\end{lemma}

\begin{definition}$(Orlicz-Campanato$ space$)$
	For any function $\phi: (0,\infty)\to(0,\infty)$, set
	$$\mathcal{L}(\Phi, \phi)(\R^n)=\{f\in L^1_{loc}(\R^n): \|f\|_{\mathcal{L}(\Phi,\phi)}:=\sup\limits_B \bigg\|f-\fint_B f\bigg\|_{\Phi,\phi,B}<\infty\}.$$
\end{definition}
\subsection{Maximal function}
We present the definitions and properties of several related maximal operators in this subsection for later use.
\begin{definition}
	For any function $\psi: (0,\infty)\to (0,\infty)$, define
	$$M_\psi f(x)=\sup\limits_{B(x,R)}\psi(R)\fint_{B(x,R)}|f(y)|\,dy.$$
	Given $0<\alpha<n$ and $s\geqslant1$, we define the  fractional maximal operators as
	$$M_\alpha f(x)=\sup\limits_{B(x,R)}R^\alpha
	\fint_{B(x,R)}|f(y)|\,dy$$
	and
	$$M_{\alpha,s}f(x)=\sup\limits_{B(x,R)}R^\alpha\bigg(\fint_{B(x,R)}|f(y)|^s\,dy\bigg)^{\frac{1}{s}}$$
	while the sharp maximal operator is defined by
	$$M^\sharp f(x)=\sup\limits_{B(x,R)}\fint_{B(x,R)}\bigg|f(y)-\fint_{B(x,R)} f\bigg|\,dy,$$
	where each supremum above is taken over all balls $B(x,R)\subset \R^n$.
\end{definition}

\begin{lemma}\label{M Bds}\rm(\cite{SAN21}) \em
	Let $\Phi,\Psi\in \Delta_2\cap \nabla_2$, $\phi\in\mathcal{G}^{dec}$ and $\psi:(0,\infty)\to (0,\infty)$. Assume that $\lim\limits_{r\to \infty} \phi(r)=0$ and there exists $C>0$ satisfying
	$\psi(t)\leqslant C\psi(s)$ for any $t<s$. If there exists $D>0$ such that $$\phi(R)\Phi^{-1}(\phi(R)) \leqslant D \Psi^{-1}(\phi(R))$$ for any $R>0$, then $M_\psi$ is bounded from $L(\Phi,\phi)(\R^n)$ to $L(\Psi,\phi)(\R^n)$.
\end{lemma}

\begin{lemma}\rm(\cite[Corollary 6.3]{SAN21})\em\label{em M1}
	Let $\Phi\in \Delta_2$ and $\phi\in$ $\mathcal{G}^{dec}$. If $f\in L^1_{loc}(\R^n)$ satisfies $\lim\limits_{r\to\infty}\fint_{B(r)} f(x)\,dx=0$ and $\phi$ satisfies \rm(\ref{phi})\em , then
	$\|M^\sharp f\|_{\mathcal{L}(\Phi, \phi)}\sim \|f\|_{L(\Phi,\phi)}$.
\end{lemma}

\section{The weighted estimates for  the commutator $[b, T_{\Omega,\alpha}^m]$}
Given any $\alpha\in[0,n)$, $ m\geqslant1$, $ \sum_{i=1}^m\frac{n}{q_i}=n-\alpha$ and suitable function $f$, recall the definition of $T_{\Omega,\alpha}^m f$,
\begin{align}\label{def}
	T_{\Omega,\alpha}^m f(x)=\int_{\R^n}\frac{\Omega_1(x-A_1 y)}{|x-A_1 y |^{\frac{n}{q_1}}} \cdots \frac{\Omega_m(x-A_m y)}{|x-A_m y |^{\frac{n}{q_m}}} f(y)dy,
\end{align}
and we  denote the kernel of  $T_{\Omega,\alpha}^m$ by $K(x,y)$ through out this paper.
In this section we make the following assumptions:
\begin{align}\label{A}
	A_i\ \  is\ \ invertible,\ \  A_i-A_j\  \ is\ invertible \ \ and\ \  \omega(A_ix) \lesssim \omega(x)\ a.e.,
\end{align}
$p_i>q_i>0$ and set $s\geqslant 1$ such that $\frac{1}{p_1}+\cdots+\frac{1}{p_m}+\frac{1}{s}=1$.
Besides, assume that $\Omega_i$ is homogeneous of degree 0 on $\R^n$, $\Omega_i\in L^{p_i}(S^{n-1})$  and
\begin{align}\label{w}
	\int_{0}^1 \frac{\omega_{i,p_i}(\delta)}{\delta} \,d\delta<\infty\ \ where \  \ \omega_{i,p_i}(\delta):=\sup_{\|\rho\|<\delta}\|\Omega_i\big(\rho(\cdot)\big)-\Omega_i(\cdot)\|_{L^{p_i}(S^{n-1})}
\end{align}
defines the $L^{p_i}$-modulus of continuity for $\Omega_i$.

\begin{definition}
	A locally integrable function $b$ is said to belong to $BMO(\R^n)$ if and only if
	\begin{align*}
		\|b\|_{BMO(\R^n)}:=\sup\limits_B\frac{1}{|B|}\int_B |b(y)-\fint_B b|\,dy,
	\end{align*}
	where the supremum is taken over all balls $B\subset \R^n$.
\end{definition}

In the remainder of this section, $b$ is a $BMO$ function unless otherwise stated. Now we consider the behavior of the commutator $[b,T_{\Omega,\alpha}^m]$ defined by
$$[b,T_{\Omega,\alpha}^m]f(x)=\int_{\R^n}\big(b(x)-b(y)\big)K(x,y)f(y)dy$$
for suitable function $f$ and the main results in weighted spaces are Theorem \ref{T1} and Theorem \ref{RESULT 1} below.

Let us introduce two lemmas as the  preparations for Theorem \ref{T1}.

\begin{lemma}\label{W}\rm(\cite[Lemma 3.12]{BMMST20})\em
	Fix $1<r<\infty$ and $1<\eta<\infty$. For any $\omega^\eta\in \mathcal{A}_r$ and $b\in BMO(\R^n)$,
	\begin{align*}
		[\omega e^{\lambda b}]_{\mathcal{A}_r}\leqslant[\omega^\eta]^{\frac{1}{\eta}}_{\mathcal{A}_r} 4^{|\lambda|\|b\|_{BMO(\R^n)}} \ \ if \ \ |\lambda|\leqslant \frac{min\{1,r-1\}}{\eta'\|b\|_{BMO(\R^n)}},
	\end{align*}
	where $[\omega ]_{\mathcal{A}_r}$ denotes the constant $\sup\limits_{B} \bigg(\frac{1}{|B|}\int_{B}\omega(x)\,dx\bigg)\bigg(\frac{1}{|B|}\int_{B}\omega^{1-r'}(x)\,dx\bigg)^{r-1}$.
\end{lemma}

\begin{lemma}\rm(\cite[Theorem 3.3]{RU14})\em\label{b}
	Let $\alpha\in [0,n)$, $ m>1$, $\{p_i,q_i,s\}^m_{i=1}$ be constants connected to $T_{\Omega,\alpha}^m$ as before, $s<p<\frac{n}{\alpha}$ and $\frac{1}{q}=\frac{1}{p}-\frac{\alpha}{n}$. Suppose that $A_i,\omega$ and $\Omega_i$ satisfy \rm(\ref{A})\em and \rm(\ref{w})\em. Then for any $\omega^s\in \mathcal{A}(\frac{p}{s}.\frac{q}{s})$ and $f\in L^\infty_c(\R^n)$,
	\begin{align*}
		\bigg(\int_{\R^n}|T_{\Omega,\alpha}^m f(x)|^q \omega^q(x)dx\bigg)^{\frac{1}{q}}\lesssim\bigg(\int_{\R^n}|f(x)|^p \omega^p(x)dx\bigg)^{\frac{1}{p}}.
	\end{align*}
\end{lemma}

\begin{theorem}\label{T1}
	Let $\alpha\in [0,n)$, $ m>1$, $\{p_i,q_i,s\}^m_{i=1}$ be constants connected to $T_{\Omega,\alpha}^m$ as before, $s<p<\frac{n}{\alpha}$ and $\frac{1}{q}=\frac{1}{p}-\frac{\alpha}{n}$. Suppose that $A_i,\omega$ and $\Omega_i$ satisfy \rm(\ref{A})\em and \rm(\ref{w})\em.
	Then for any $\omega^s\in \mathcal{A}(\frac{p}{s}.\frac{q}{s})$ and $b\in BMO(\R^n)$, $[b,T_{\Omega,\alpha}^m]$ can be extended to a bounded operator from $L^p_{\omega^p}(\R^n)$ to $L^q_{\omega^q}(\R^n)$.
\end{theorem}

\begin{proof}
	For any $f\in L^\infty_c(\R^n)$, $b\in L^\infty(\R^n)$, $x\in \R^n$ and $z\in \C^n$, set
	\begin{align*}
		G(z)=e^{zb(x)}T_{\Omega,\alpha}^m (fe^{-zb(\cdot)})(x).
	\end{align*}
	$G(z)$ turns out to be analytic at $0$ and by calculation,
	\begin{align}\label{trans}
		\frac{\partial G}{\partial z}\bigg|_{z=0}=[b,T_{\Omega,\alpha}^m]f(x).
	\end{align}
	On the other hand, through the method used in \cite{ABKP93,BMMST20, N10} , the Cauchy integral formula for $G(z)$ shows that
	\begin{align}\label{Cauchy}
		\frac{\partial G}{\partial z}\bigg|_{z=0}=\frac{1}{2\pi i}\int_{|z|=\epsilon} \frac{G(z)}{(z-0)^2} dz
	\end{align}
	for suitable $\varepsilon>0$.
	
	Combining (\ref{trans}), (\ref{Cauchy}) and the Minkowski inequality, it follows that
	\begin{align*}
		\|[b,T_{\Omega,\alpha}^m]f\|_{L^q_{\omega^q}}
		&=\|\frac{\partial G}{\partial z}\bigg|_{z=0}\|_{L^q_{\omega^q}}(x)\\
		&\leqslant \frac{1}{2\pi} \int_{|z|=\epsilon} \frac{\|e^{zb(x)}T_{\Omega,\alpha}^m (fe^{-zb(\cdot)})(x)\|_{L^q_{\omega^q}(x)}}{|z|^2} dz\\
		&\leqslant \frac{1}{\varepsilon} \sup_{|z|=\varepsilon }\|e^{zb(x)}T_{\Omega,\alpha}^m (fe^{-zb(\cdot)})(x)\|_{L^q_{\omega^q}(x)}.
	\end{align*}
	Then according to \cite[Lemma 3.12]{N10} and Lemma \ref{W}, $(\omega e^{b\cdot Re(z)})^s \in \mathcal{A}(\frac{p}{s},\frac{q}{s})$ holds for any $|z|=\varepsilon=\frac{min\{1,\frac{\frac{q}{s}}{\frac{p}{s}'}\}}{qs'\|b\|_{BMO(\R^n)}}$ and therefore by Lemma \ref{b} we have
	\begin{align*}
		\|e^{zb(x)}T_{\Omega,\alpha}^m (fe^{-zb(\cdot)})(x)\|_{L^q_{\omega^q}(x)}&=\|T_{\Omega,\alpha}^m (fe^{-zb(\cdot)})(x)\|_{L^q_{{(\omega e^{b\cdot Re(z)})}^q}(x)}\\
		&\lesssim \|fe^{-zb(\cdot)}\|_{L^p_{{(\omega e^{b\cdot Re(z)})}^p}}=\|f\|_{L^p_{\omega^p}(\R^n)}.
	\end{align*}
	Now
	\begin{align*}
		\|[b,T_{\Omega,\alpha}^m]f\|_{L^q_{\omega^q}(\R^n)}\lesssim \|b\|_{BMO(\R^n)}\|f\|_{L^p_{\omega^p}(\R^n)}
	\end{align*}
	for any $f\in L^\infty_c(\R^n)$ under the hypothesis that $b\in L^\infty(\R^n)$.
	
	For a general $BMO$ function $b$ and any $N\in \N^+$, set $$b^N(x)=\begin {cases}
	b(x)   & \text{if \, $-N\leqslant\ b(x) \leqslant N$,} \\
	N & \text {if \, $b(x)>N$,} \\
	-N & \text {if \, $b(x)<N$.}
	\end {cases}$$
	{Claim }:
	\begin{align*}
		\|b^N\|_{BMO(\R^n)}\leqslant 2\|b\|_{BMO(\R^n)} \ for \ any\  N\in \N^+.
	\end{align*}
	Indeed, given any ball $B\subset \R^n$, assume that $\fint_{B}b^N\leqslant\fint_{B}b$.
	Thus
    \begin{align*}
    0\leqslant \fint_{B}b-\fint_{B}b^N
    &=\frac{1}{|B|}\int_{B\cap \{y:\,b(y)\geqslant\fint_{B}b^N\}}\bigg(b(y)-b^N(y)\bigg)\, dy\\
    &\ \ \ \ +\frac{1}{|B|}\int_{B\cap \{y:\,b(y)\leqslant\fint_{B}b^N\}}\bigg(b(y)-b^N(y)\bigg)\, dy,
   \end{align*}
 where the first item in the righthand side  is nonnegative and the second one is nonpositive.
	Then we obtain
	$$\int_{B\cap \{y:\,b(y)\geqslant\fint_{B}b^N\}} \bigg(\fint_{B}b-\fint_{B}b^N\bigg)\,dy\leqslant \int_{B\cap \{y:\,b(y)\geqslant\fint_{B}b^N\}} \bigg(b(y)-b^N(y)\bigg)\,dy ,$$
	which immediately shows that
	\begin{align}\label{item 1}
		\int_{B\cap \{y:\,b^N(y)>\fint_{B}b^N\}} \bigg(b^N(y)-\fint_{B}b^N\bigg)\,dy\leqslant \int_{B\cap \{y:\,b(y)\geqslant\fint_{B}b^N\}} \bigg(b(y)-\fint_{B}b\bigg)\,dy.
	\end{align}
	Besides,
	\begin{align}\label{item 2}
		\int_{B\cap \{y:\,b^N(y)<\fint_{B}b^N\}} \bigg( \fint_{B}b^N-b^N(y)\bigg)\,dy\leqslant \int_{B\cap \{y:\,b(y)<\fint_{B}b^N\}} \bigg(\fint_{B}b-b(y)\bigg)\,dy.
	\end{align}
	Combining (\ref{item 1}) and (\ref{item 2}), it follows that
	$$\frac{1}{|B|}\int_{B}\big|b^N(y)-\fint_{B}b^N\big|\, dy\leqslant 2\|b\|_{BMO(\R^n)}.$$
	Similar proof works for the condition $\fint_{B}b^N\geqslant\fint_{B}b$ and the proof of the Claim is complete.
	
	Note that $b^N\to b\ a.e.$ and $b^Nf\to bf\ in\ L^p(\R^n)$, then by Fatou’s lemma and the Claim we deduce that
	\begin{align*}
		\|[b,T_{\Omega,\alpha}^m]f\|_{L^q_{\omega^q}(\R^n)}\leqslant \lim_{N\to \infty} \|[b,T_{\Omega,\alpha}^m]f\|_{L^q_{\omega^q}(\R^n)} \lesssim \|b\|_{BMO(\R^n)}\|f\|_{L^q_{\omega^q}(\R^n)},
	\end{align*}
	which is the desired conclusion.
\end{proof}

\begin{corollary}
	Let $1<s<p<\frac{n}{\alpha}$,  $\frac{1}{q}=\frac{1}{p}-\frac{\alpha}{n}$ and $T$ be a linear operator on $L^r_{\omega}(\R^n)$ for any $\omega \in \mathcal{A}_\infty$ and $r\geqslant1$. If
	\begin{align*}
		\bigg(\int_{\R^n}|T f(x)|^q \omega^q(x)dx\bigg)^{\frac{1}{q}}\lesssim\bigg(\int_{\R^n}|f(x)|^p \omega^p(x)dx\bigg)^{\frac{1}{p}}
	\end{align*}	
	for any $\omega^s\in \mathcal{A}(\frac{p}{s}.\frac{q}{s})$ and $f\in L^\infty_c(\R^n)$,  then for any  $\omega^s\in \mathcal{A}(\frac{p}{s}.\frac{q}{s})$ and $b\in BMO(\R^n)$, $[b,T]$ can be extended to a bounded operator from $L^p_{\omega^p}(\R^n)$ to $L^q_{\omega^q}(\R^n)$.
\end{corollary}

As mentioned in the introduction, the $H^p(\R^n)-L^{\frac{np}{n-\alpha p}}(\R^n)$ boundedness of $T_{\Omega,\alpha}^m$ and part of the behaviors of $T_{\Omega,\alpha}^m$ as well as the commutator of $T_{\Omega,\alpha}^m$ with $BMO$ function in weighted spaces have been investigated in \cite{IR18,IV23,RU12, RU17}. We are eager to move forward with this topic in the weighted Hardy space. It is necessary to prepare more before begining our debate. We start by introducing a new set of $BMO$ functions and giving some explanations for them.

\begin{definition}\label{BMOA}
	Let $A$ be an arbitrary invertible matrix, a $BMO$ function $b$ is said to be in $BMO^A(\R^n)$ if and only if $$\sup\limits_{B} \bigg|\fint_{B} b-\fint_{A(B)} b\bigg|<\infty,$$
	where the supremum is taken over all balls $B\subset \R^n$ and $A(B)$ denotes the set $\{Ax: x\in B\}$. If $A=I$, then $BMO^A(\R^n)=BMO(\R^n)$.
\end{definition}

\begin{example}
	$L^\infty(\R^n)\subset BMO^A(\R^n)$ for any invertible matrix $A$. Moreover, for any  $0<\alpha\leqslant1$, $(\log |x|)^\alpha$ and  $\big| \log|x|\big|^\alpha$ are in $BMO^A(\R^n)$ for any invertible matrix $A$.
\end{example}

Here we take $\log |x|$ for instance. Let $A$ be an invertible matrix and $b\in BMO(\R^n)$, it is obvious for us to have the following equivalent characterization for functions in $BMO^A(\R^n)$:
$$\sup\limits_{B(x,R)} \bigg|\fint_{B(x,R)} b-\fint_{B(Ax,R)} b\bigg|<\infty.$$
Since $\log |x|\in BMO(\R^n)$, then it suffices to find a $C>0$ such that
$$\bigg|\fint_{B(x,R)} \log |x|-\fint_{B(Ax,R)} \log |x|\bigg|<C$$ for any ball $B=B(x,R)\subset \R^n$.
For the case $|x|<NR$ where $N=N(A)$ is a positive integer bigger than $max\{10,10\|A^{-1}\|\}$, we have
\begin{align*}
	\bigg|\fint_{B(x,R)} \log |x|-\fint_{B(Ax,R)} \log |x|\bigg|
	&\leqslant \log(1+\frac{|Ax-x|}{R})\cdot \|\log |x|\|_{BMO(\R^n)}\\
	&\lesssim \log \big((\|A\|+1)N \big) \big\|\log |x|\big\|_{BMO(\R^n)}.
\end{align*}
For the other case, check that $$\log\big(|x|-R\big)\leqslant\fint_{B(x,R)} \log |x|\leqslant \log\big(|x|+R\big),$$
$$\log\big(|Ax|-R\big)\leqslant\fint_{B(Ax,R)} \log |x|\leqslant \log\big(|Ax|+R\big)$$
and thus
\begin{align*}
	\bigg|\fint_{B(x,R)} \log |x|-\fint_{B(Ax,R)} \log |x|\bigg|
	&\leqslant max\bigg\{\log\frac{\|A\||x|+R}{|x|-R},\log\frac{|x|+R}{\|A\||x|-R}\bigg\}\\
	&\lesssim \log \frac{11}{9}(\|A\|+\|A\|^{-1})
\end{align*}
which gives $\log |x|\in BMO^A(\R^n)$.

Furthermore, we can give a necessary condition for functions in the set $BMO^A(\R^n).$ Hereafter, $z\in \R^n$ is said to be a critical point of a $BMO$ function $b$ if and only if  we can find a ball $B\subset B(z,\varepsilon)$ such that $\fint_B b>N$ for any $\varepsilon >0$ and $N>0$.
\begin{proposition}
	If $b\in BMO^A(\R^n)$, then it must lie in one of the following situations:
	
	\begin{enumerate}[(i)]
		\item b has no critical points; (Example: $L^\infty(\R^n)$)
		\item b has the only critical point $0$; (Example: $\log|x|$)
		\item if $0\neq z\in \R^n$ is a critical point of $b$, then so as $\{A^j z\}_{j=1}^\infty$. \\
		\rm(Example: $\log|x-z|+\log|x-Az|+\cdots+\log|x-A^{K(A)-1}z|$,\\ where $A$ is a nilpotent matrix and $K(A)$ is the smallest integer $K$ such that $A^{K}=I$. )
	\end{enumerate}
\end{proposition}
\begin{proof}
	If $b\in BMO^A(\R^n)$ lies out of the above situations, then there must exist a critical point $z$ of $b$ while $Az$ is not a critical point of $b$. Consequently we can find $\varepsilon_0>0$ and $N_0>0$ such that $\fint_B b\leqslant N_0$ for any ball $B\subset B(Az,\varepsilon_0)$.
	
	While for any $$N>N_0+(\|A\|\|A^{-1}\|)^n\cdot \|b\|_{BMO(\R^n)}\ \ and\ \ 0<\varepsilon<\frac{\varepsilon_0}{2\|A\|},$$ there exists a ball $B(x,R)\subset B(z,\varepsilon)$ such that $\fint_{B(x,R)} b>N$
	and $\fint_{B(Ax,\|A\|R)}b\leqslant N_0$ since $B(Ax,\|A\|R)\subset B(Az,\varepsilon_0)$.
	It follows that
	\begin{align*}
		\fint_{A(B(x,R))} b
		&\leqslant \fint_{B(Ax, \|A\|R)} b+\bigg|\fint_{A(B(x,R))} b-\fint_{B(Ax, \|A\|R)} b\bigg|\\
		&\leqslant N_0+\frac{|B(Ax, \|A\|R)|}{|A(B(x,R))|}\cdot\|b\|_{BMO(\R^n)}\\
		&\leqslant N_0+\frac{|B(Ax, \|A\|R)|}{|B(Ax, \|A^{-1}\|^{-1}R)|}\cdot\|b\|_{BMO(\R^n)}\\
		&\leqslant N_0+(\|A\|\|A^{-1}\|)^n \cdot\|b\|_{BMO(\R^n)}
	\end{align*}
	and then $$\fint_{B(x,R)} b-\fint_{A(B(x,R))} b>N-N_0-(\|A\|\|A^{-1}\|)^n \cdot\|b\|_{BMO(\R^n)},$$
	which clearly conflicts with $b\in BMO^A(\R^n)$.
\end{proof}

To obtain the boundedness of certain operator in weighted Hardy spaces, we usually need taking advantage of the definitions of atoms which are defined as follows.

\begin{definition}
	Given any $0<p\leqslant 1\leqslant q\leqslant \infty$ and $\omega \in \mathcal{A}_q$, a $(H^p_{\omega}-q)$ atom is a measurable function a: $\R^n \to \R^1$ supported in a ball B such that
	$$\|a\|_{L^q_\omega(\R^n)}\leqslant [\omega(B)]^{\frac{1}{q}-\frac{1}{p}}\ \ and\ \ \int_{\R^n}  a(x)\,dx=0.$$
\end{definition}

\begin{definition}
	Given any $0<p\leqslant 1\leqslant q\leqslant \infty$ and $\omega \in \mathcal{A}_q$, a $\omega-(p,q,d)$ atom where $d\geqslant \left\lfloor n(\frac{q_\omega}{p}-1)\right\rfloor$ is a measurable function a: $\R^n \to \R^1$ supported in a ball B such that
	$$\|a\|_{L^q(\R^n)}\leqslant \frac{|B|^\frac{1}{q}}{[\omega(B)]^\frac{1}{p}}\ \ and\ \ \int_{\R^n} x^\alpha a(x)\,dx=0\  for\  any \ multi-indices \ |\alpha
	|\leqslant d.$$
\end{definition}

\begin{definition}(\cite[Definition 1]{HK21})
	Let $0<p\leqslant1$ and $\omega \in \mathcal{A}_\infty$ satisfy $\int_{\R^n}\frac{\omega(x)}{(1+|x|)^{np}}\, dx<\infty$. A locally integrable function $b$ is said to belong to $BMO_{\omega,p}(\R^n)$ if
	\begin{align*}
		\|b\|_{BMO_{\omega,p}(\R^n)}:=\sup\limits_B\bigg\{\bigg[\frac{1}{\omega(B)}\int_{B^c}\frac{\omega(x)}{|x-x_B|^{np}}\,dx\bigg]^{\frac{1}{p}}\int_B\bigg|b(y)-\fint_B b\bigg|\,dy\bigg\}<\infty,
	\end{align*}
	where the supremum is taken over all balls $B=B(x_B,r_B)\subset \R^n$.
\end{definition}

In what follows, we introduce the lemmas which guarantee our proof of Theorem \ref{RESULT 1}.

\begin{lemma}\rm(\cite{LKY16})\em
	Let $p\geqslant 1$, $\omega \in \mathcal{A}_\infty$ and $b\in BMO(\R^n)$, then for any ball $B\subset \R^n$,
	$$\bigg(\frac{1}{\omega(B)}\int_B \bigg|b(y)-\fint_B b\bigg|^p\omega(y)\,dy\bigg)^{\frac{1}{p}}\lesssim \|b\|_{BMO(\R^n)}$$
\end{lemma}

\begin{lemma}\rm(\cite[Theorem 1.1]{HW22})\em \label{estimate 1}
	For any $0<\alpha<1$, $\frac{n}{n-\alpha}\leqslant p<1$, $\frac{1}{q}=\frac{1}{p}-\frac{\alpha}{n}$, $\omega^\frac{n}{n-\alpha} \in \mathcal{A}_1$ with $\int_{\R^n} \frac{\omega^p(x)}{(1+|x|)^{np}}\,dx$ and $b\in BMO_{\omega^p,p}(\R^n)$,
	$$\|(b-\fint_B b)a\|_{H^p_{\omega^p}(\R^n)}\lesssim \|b\|_{BMO_{\omega^p,p}(\R^n)}$$
	for any $(H^p_{\omega^p}-\infty)$ atom $a$ supported in a ball $B$.
\end{lemma}

\begin{lemma}\rm(\cite{BLYZ08,HW22, MSV08})\em\label{extend}
	Let $0<p\leqslant1$ and $\omega\in \mathcal{A}_\infty$. If a linear operator T satisfies one of the following conditions:\\
	(1) T is well-defined for all the finite combinations of  continuous $(H^p_{\omega^p}-\infty)$ atoms and
	$$\sup\big\{\|Ta\|_{L^q_{\omega^q}(\R^n)} : a\ is\ a\ continuous\  (H^p_{\omega^p}-\infty) \ atom\big\}<\infty;$$
	(2) T is well-defined for all the finite combinations of  $(H^p_{\omega^p}-q)$ atoms where $q>q_\omega$ and
	$$\sup\big\{\|Ta\|_{L^q_{\omega^q}(\R^n)} : a\ is\ a\ \  (H^p_{\omega^p}-q) \ atom\big\}<\infty.$$
	Then $T$ can be uniquely and continuously extended to a bounded linear operator from $H^p_{\omega^p}(\R^n)$ to $L^q_{\omega^q}(\R^n)$.
\end{lemma}

\begin{lemma}\rm(\cite[Theorem 2.9]{R20})\em\label{char}
	Let $\omega \in \mathcal{A}_\infty$ with critical index $q_w$ and $r_w$, $0<p\leqslant 1\leqslant max\{1,p(\frac{r_\omega}{r_{\omega-1}})\}< q\leqslant \infty$ and $d\geqslant \left\lfloor n(\frac{q_\omega}{p}-1)\right\rfloor$. For any $f\in H^p_{\omega^p}(\R^n)\cap L^q(\R^n)$, there exists a sequence of $\omega^p-(p,q,d)$ atoms $\{a_j\}_{j=1}^\infty$ and a sequence of positive constants $\{\lambda_j\}_{j=1}^\infty$ such that $\sum_{j=1}^\infty \lambda_j a_j\to f$ in $L^q(\R^n)$ with $\sum_{j=1}^\infty|\lambda_j|^p\lesssim \|f\|_{H^p_{\omega^p}(\R^n)}$.
\end{lemma}

\begin{lemma}\rm(\cite[Theorem 1.1]{RU17})\em \label{Boundedness 1}
	Let $\alpha\in [0,n)$, $m>1$, $\sum_{i=1}^m\frac{n}{q_i}=n-\alpha$, $\frac{n}{n+1}<p<1$ and $\frac{1}{q}=\frac{1}{p}-\frac{\alpha}{n}$. Suppose that $A_i,\Omega_i$ and $\omega$ satisfy $(\ref{A}), (\ref{w})$ and $\int_{0}^1 \frac{\omega_{p_i}(\delta)}{\delta^{1+n/q_i'}}\, d\delta<\infty$ respectively.
	Then for any $\omega^{\frac{sn}{n-\alpha s}} \in \mathcal{A}_1 $, there exists $C>0$ such that $\|Ta\|_{L^q_{\omega^q}(\R^n)}\leqslant C $ for any $(H^p_{\omega^p}-\infty)$ atom $a$.
\end{lemma}

\begin{theorem}\label{RESULT 1}
	Let $\alpha\in [0,1)$, $m>1$, $\{p_i,q_i,s\}^m_{i=1}$ be constants connected to $T_{\Omega,\alpha}^m$ as before, $s<p<\frac{n}{\alpha}$ and $\frac{1}{q}=\frac{1}{p}-\frac{\alpha}{n}$. Suppose that $A_i,\Omega_i$ and $\omega$ satisfy $(\ref{A}), (\ref{w})$ and $\int_{0}^1 \frac{\omega_{i,p_i}(\delta)}{\delta^{1+n/q_i'}} \,d\delta<\infty$ respectively.
	Then for any $\omega^{\frac{sn}{n-\alpha s}} \in \mathcal{A}_1 $ and $b\in \cap_i BMO^{A_i}(\R^n)\cap BMO_{\omega^p,p}(\R^n)$,  the commutator $[b,T_{\Omega,\alpha}^m]$ can be extended to a bounded operator from $H^p_{\omega^p}(\R^n)$ to $L^q_{\omega^q}(\R^n)$.
\end{theorem}

\begin{proof}
	We first observe that for any continuous $(H^p_{\omega^p}-\infty)$ atom $a$ supported in a ball denoted by $B=B(x_{a},R)$,
	$$\|[b,T_{\Omega,\alpha}^m]a\|_{L^q_{\omega^q}(\R^n)}
	\leqslant \|T_{\Omega,\alpha}^m (b-\fint_B b)a\|_{L^q_{\omega^q}(\R^n)}+\|(b-\fint_B b)T_{\Omega,\alpha}^m a\|_{L^q_{\omega^q}(\R^n)}.$$
	Thanks to Lemma \ref{extend}, we are reduced to show that the two items above are bounded uniformly.\\
	For the first one, we start with the observation that
	there exists a positive constant $C$ such that $\|T_{\Omega,\alpha}^m a'\|_{L^q_{\omega^q}(\R^n)}\leqslant C$ for any $\omega^p-(p,r,d)$ atom $a'$ supported on a ball denoted by $B(x_{a'},R')$.
	Set $M=\max\limits_{i=1,2,\cdots,m}\|A_i\|$ , $B'_i=B(A_i x_{a'},MR')$ and consequently
	\begin{align}\label{item 3}
		&\nonumber\bigg(\int_{B'_i} |T_{\Omega,\alpha}^m a'(x)|^q \omega^q(x) \,dx \bigg)^{\frac{1}{q}}\\
		&\nonumber\leqslant \bigg(\int_{B'_i} |T_{\Omega,\alpha}^m a'(x)|^{\frac{nr}{n-\alpha r}} \,dx \bigg)^{\frac{1}{r}-\frac{\alpha}{n}}\bigg(\int_{B'_i} \omega^{\frac{1}{\frac{1}{p}-\frac{1}{r}}} (x)\,dx \bigg)^{{\frac{1}{p}}-\frac{1}{r}}\\
		&\lesssim\|a'\|_{L^r(\R^n)}\bigg[\omega^{\frac{1}{\frac{1}{p}-\frac{1}{r}}}(B'_i)\bigg]^{{\frac{1}{p}}-\frac{1}{r}}
		\leqslant\frac{|B(x_{a'},R')|^\frac{1}{r}}{[\omega\big(B(x_{a'},R')\big)]^\frac{1}{p}}\bigg[\omega^{p\cdot\frac{1}{1-\frac{p}{r}}}(B'_i)\bigg]^{{\frac{1}{p}}-\frac{1}{r}}.
	\end{align}
	If $r>\frac{pqs}{q(2s-1)-ps}$, then $\omega^p\in RH_{\frac{qs}{p(s-q(s-1))}}\subset RH_{\frac{1}{1-\frac{p}{r}}}$ and, in consequence,
	$$\bigg[\omega^{p\cdot\frac{1}{1-\frac{p}{r}}}(B'_i)\bigg]^{{\frac{1}{p}}-\frac{1}{r}}\lesssim \frac{\big[\omega^p(B'_i)\big]^{\frac{1}{p}}}{|B'_i|^{\frac{1}{r}}}.$$
	Now (\ref{item 3}) turns out to be bounded uniformly.
	According to the proof of \cite[Theorem 1.1]{RU17},
	\begin{align*}
		&\bigg(\int_{({\cup B'_i})^c}|T_{\Omega,\alpha}^m a'(x)|^q \omega^q(x) \,dx \bigg)^{\frac{1}{q}}\\
		&\lesssim \int_{B(x_{a'},R')} |a'(z)| \,dz\cdot R^{\alpha-\frac{n}{s}}\cdot\sum\limits_{i=1}^m \bigg( \omega^{\frac{qs}{s-q(s-1)}}\big(B(x_{a'},R')\big)\bigg)^{\frac{1}{q}-\frac{1}{s'}}\\
		&\lesssim \|a'\|_{L^r(\R^n)} \|\chi_{B(x_{a'},R')}\|_{L^{r'}(\R^n)}\cdot R^{\alpha-\frac{n}{s}}\cdot\sum\limits_{i=1}^m \bigg( \omega^{\frac{qs}{s-q(s-1)}}\big(B(x_{a'},R')\big)\bigg)^{\frac{1}{q}-\frac{1}{s'}}\\
		&\lesssim R^n\cdot [\omega^p(B(x_{a'},R'))]^{-\frac{1}{p}}\cdot R^{\alpha-\frac{n}{s}}\cdot\sum\limits_{i=1}^m \bigg( \omega^{\frac{qs}{s-q(s-1)}}\big(B(x_{a'},R')\big)\bigg)^{\frac{1}{q}-\frac{1}{s'}}\lesssim 1.
	\end{align*}
	
	Since $(b-\fint_B b)a\in L^r(\R^n)$ for some $r>max\{1,p\frac{r_{\omega^p}}{r_{\omega^p}-1},\frac{pqs}{q(2s-1)-ps}\}$, Lemma \ref{char} shows that there exists a sequence of $\omega^p-(p,r,d)$ atoms $\{a_j\}_{j=1}^\infty$ and a sequence of positive constants $\{\lambda_j\}_{j=1}^\infty$ such that $$\sum_{j=1}^\infty \lambda_j a_j\to (b-\fint_B b)a \ in \ L^r(\R^n) \ \ and\ \ \sum_{j=1}^\infty |\lambda_j|^p\lesssim \|(b-\fint_B b)a\|^p_{H^p_{\omega^p}(\R^n)}.$$
	Besides, $T_{\Omega,\alpha}^m$ is bounded from $L^r(\R^n)$ to $L^{\frac{nr}{n-\alpha r}}(\R^n)$ by \cite{GU99} which implies  $$T_{\Omega,\alpha}^m(\sum_{j=1}^N \lambda_j a_j)\to T_{\Omega,\alpha}^m\big((b-\fint_B b)a\big) \ a.e.\ as\ N\to \infty$$ and hence Fatou's lemma, our observation above and Lemma \ref{estimate 1} yield
	\begin{align*}
		\|T_{\Omega,\alpha}^m \big((b-\fint_B b)a\big)\|_{L^q_{\omega^q}(\R^n)}
		&\leqslant \sum_{j=1}^\infty |\lambda_j|\|T_{\Omega,\alpha}^m a_j\|_{L^q_{\omega^q}(\R^n)} \lesssim \bigg(\sum_{j=1}^\infty |\lambda_j|^p\bigg)^{\frac{1}{p}}\\
		&\lesssim \|(b-\fint_B b)a\|_{H^p_{\omega^p}(\R^n)} \lesssim \|b\|_{BMO_{\omega^p,p}(\R^n)} .
	\end{align*}
	
	We are now in a position to consider the second item. Recall the definition of $BMO^A(\R^n)$ and choose some positive constants $\{C(A_i,b)\}_{i=1}^m$ such that $|\fint_{B_0} b-\fint_{A_i(B_0)} b|\leqslant C(A_i,b)$ for any ball $B_0\subset \R^n$. Choose $M'=\max\limits_{i}\|A^{-1}_i\|$, $B_i=B(A_i x_{a},M\sqrt n R)$, $\tilde{B}_i =A_i(B(x_a, MM'\sqrt n R))$, $\{p^*, q^*\}$ satisfying the assumptions in \cite[Theorem 3.4]{RU14} and then it follows that	
	\begin{align*}
		&\bigg(\int_{ B_i}|(b(x)-\fint_B b) T_{\Omega,\alpha}^m a(x))|^q \omega^q(x) \,dx \bigg)^{\frac{1}{q}}\\
		&\leqslant \bigg(\int_{\tilde{B}_i}|(b(x)-\fint_B b) T_{\Omega,\alpha}^m a(x)|^q \omega^q(x) \,dx \bigg)^{\frac{1}{q}}\\
		&\leqslant \bigg|\fint_B b-\fint_{\tilde{B}_i} b\bigg|\|T_{\Omega,\alpha}^m a\|_{L^q_{\omega^q}(\R^n)}
		+\bigg(\int_{\tilde{B}_i}|(b(x)-\fint_{\tilde{B}_i} b) T_{\Omega,\alpha}^m a(x))|^q \omega^q(x) \,dx \bigg)^{\frac{1}{q}}\\
		&\lesssim \bigg|\fint_B b-\fint_{B(x_a, MM'\sqrt n R)}b\bigg|+\bigg|\fint_{B(x_a, MM'\sqrt n R)}b-\fint_{\tilde{B}_i} b\bigg|\\
		&\ \ \ \ +\|T_{\Omega,\alpha}^m a\|_{L^{q^*}_{\omega^{q^*}}(\R^n)}\bigg(\int_{\tilde{B}_i}|b(x)-\fint_{\tilde{B}_i} b|^{q(\frac{q^*}{q})'} \,dx \bigg)^{\frac{1}{q}-\frac{1}{q^*}}\\
		&\lesssim \|a\|_{L^{p^*}_{\omega^{p^*}}(\R^n)}\cdot\|b\|_{BMO(\R^n)}\cdot|B_i|^{{\frac{1}{q}-\frac{1}{q^*}}}+\|b\|_{BMO(\R^n)}+C(A_i,b)\\
		&\lesssim R^{\frac{n}{q}-\frac{n}{q^*}}\cdot \bigg[\omega^p(B)\bigg]^{-\frac{1}{p}}\cdot \bigg[\omega^{p^*}(B)\bigg]^{\frac{1}{p^*}}\cdot \|b\|_{BMO(\R^n)}+C(A_i,b)\\
		&\lesssim \|b\|_{BMO(\R^n)}+C(A_i,b).
	\end{align*}
	Finally we calculate that
	\begin{align*}
		&\bigg(\int_{({\cup_i B_i})^c}|\big(b(x)-\fint_B b\big) T_{\Omega,\alpha}^m a(x)|^q \omega^q(x) \,dx \bigg)^{\frac{1}{q}}\\
		&= \|\int_B(K(x,y)-K(x,x_a))(b(x)-\fint_B b)\omega(x)\chi_{({\cup_i B_i})^c}(x)a(y)dy\|_{L^{q}(\R^n)}\\
		&\leqslant\int_B\|(K(x,y)-K(x,x_a))(b(x)-\fint_B b)\omega(x)\|_{L^{q}\big({({\cup_i B_i})^c}\big)}a(y)dy\\
		&\leqslant R^n\cdot \bigg[\omega^p(B)\bigg]^{-\frac{1}{p}}\cdot\sum\limits_{i=1}^m \sum\limits_{j=1}^\infty
		\bigg(\int_{Q^i_j}|b(x)-\fint_B b|^q |K(x,y)-K(x,x_a)|^q \omega^q(x) \,dx \bigg)^{\frac{1}{q}}
	\end{align*}
	where $Q^i_j=\big\{x\in({\cup_i B_i})^c:|x-A_i x_a|=\min\limits_k |x-A_k x_a|, \ 2^{j}M\sqrt n R\leqslant|x-A_i x_a|\leqslant2^{j+1}M\sqrt n R\big\}$ and
	\begin{align*}
		&\sum\limits_{j=1}^\infty
		\bigg(\int_{Q^i_j}|b(x)-\fint_B b|^q |K(x,y)-K(x,x_a)|^q \omega^q(x) \,dx \bigg)^{\frac{1}{q}}\\
		&\lesssim \sum\limits_{j=1}^\infty \bigg\|\frac{\Omega_1(\cdot-A_1 y)}{|\cdot-A_1 y |^{\frac{n}{q_1}}}-\frac{\Omega_1(\cdot-A_1 y)}{|\cdot-A_1 x_a |^{\frac{n}{q_1}}}\bigg\|_{L^{p_1}({Q^i_j})}
		\cdot \prod\limits_{i=2}^m\bigg\|\frac{\Omega_i(\cdot-A_i y)}{|\cdot-A_i y |^{\frac{n}{q_i}}}\bigg\|_{L^{p_i}({Q^i_j})}\\
		&\ \ \ \ \cdot \|\omega(b-\fint_B b)\|_{L^{\frac{qs}{s-q(s-1)}}({Q^i_j})}.
	\end{align*}
	Estimates for the items $\bigg\|\frac{\Omega_1(\cdot-A_1 y)}{|\cdot-A_1 y |^{\frac{n}{q_1}}}-\frac{\Omega_1(\cdot-A_1 y)}{|\cdot-A_1 x_a |^{\frac{n}{q_1}}}\bigg\|_{L^{p_1}({Q^i_j})}$ and $\bigg\|\frac{\Omega_i(\cdot-A_i y)}{|\cdot-A_i y |^{\frac{n}{q_i}}}\bigg\|_{L^{p_i}({Q^i_j})}$ have already been obtained by the authors in \cite{RU17}.
	Applying their results and
	\begin{align*}
		&\|\omega(b-\fint_B b)\|_{L^{\frac{qs}{s-q(s-1)}}({Q^i_j})}\\
		&\leqslant \bigg(\int_{B(A_ix_a, 2^{j+1}M\sqrt n R)} \big(\big|b(x)-\fint_B b\big|\omega(x)\big)^{\frac{qs}{s-q(s-1)}}\,dx\bigg)^{\frac{s-q(s-1)}{qs}}\\
		&\leqslant \bigg(\int_{A_i\big(B(x_a, 2^{j+1}MM'\sqrt n R)\big)} (\big|b(x)-\fint_B b\big|\omega(x))^{\frac{qs}{s-q(s-1)}}\,dx\bigg)^{\frac{s-q(s-1)}{qs}}\\
		&\leqslant \big|\fint_B b-\fint_{A_i\big(B(x_a, 2^{j+1}MM'\sqrt n R)\big)}b\big|\bigg[w^{\frac{qs}{s-q(s-1)}}(A_i\big(B(x_a, 2^{j+1}MM'\sqrt n R)\big))\bigg]^{\frac{s-q(s-1)}{qs}}\\
		&+\bigg(\int_{A_i\big(B(x_a, 2^{j+1}MM'\sqrt n R)\big)}\big|b(x)-\fint_{A_i\big(B(x_a, 2^{j+1}MM'\sqrt n R)\big)} b\big|^{\frac{qs}{s-q(s-1)}}\,dx\bigg)^{\frac{s-q(s-1)}{qs}}\\
		&\lesssim \bigg[w^{\frac{qs}{s-q(s-1)}}(A_i\big(B(x_a, 2^{j+1}MM'\sqrt n R)\big))\bigg]^{\frac{s-q(s-1)}{qs}}\\
		&\ \ \ \ \cdot\bigg(\big|\fint_B b-\fint_{A_i\big(B(x_a, 2^{j+1}MM'\sqrt n R)\big)} b\big|+\|b\|_{BMO(\R^n)}\bigg)\\
		&\lesssim (j+1)2^{jn\frac{s-q(s-1)}{qs}}\bigg(\|b\|_{BMO(\R^n)}+C(A_i,b)\bigg)\big[w^{\frac{qs}{s-q(s-1)}}(B)\big]^{\frac{s-q(s-1)}{qs}},
	\end{align*}
	we conclude that
	$$\bigg(\int_{({\cup_i B_i})^c}\big|(b(x)-\fint_B b) T_{\Omega,\alpha}^m a(x))\big|^q \omega^q(x) \,dx \bigg)^{\frac{1}{q}}\lesssim \|b\|_{BMO(\R^n)}+\sum\limits_{i=1}^m C(A_i,b),$$
	which completes the proof.
\end{proof}

\begin{remark}
	Let $p,\delta\in(0,1]$  with $1+\frac{\delta}{n}>\frac{1}{p}$. Recall the statement of \cite[Remark 1]{HK21}, then for any given invertible matrix $A$ and $\omega\in \mathcal{A}_{(1+\frac{\delta}{n})p}$ with $\int_{\R^n}\frac{\omega(x)}{(1+|x|)^{np}}\,dx<\infty$, $Lip_c^\alpha(\R^n)\subset BMO^A(\R^n)\cap BMO_{\omega,p}(\R^n)$ for any $\alpha\geqslant \delta$.
\end{remark}

\section{Boundedness of $[b, T_{\Omega,\alpha}^m]$ from $L(\Phi,\varphi)(\R^n)\  to \ L(\Psi,\varphi)(\R^n)$}

For the generalized fractional integral operator $I_\rho$ defined by
$$ I_\rho f(x)=\int_{\R^n} \frac{\rho(x,|x-y|)}{|x-y|^n}f(y)\,dy,
$$
where $\rho: \R^n\times (0,\infty)\to(0,\infty)$ is a function such that $\int_{0}^{1} \frac{\rho(x,t)}{t}\,dt<\infty$ for each $x\in \R^n$, Nakai and his collaborators obtained the boundedness results for the commutators of $I_\rho$ with the generlized $Companato$ function in \cite{SAN21, YN22}. They showed that  $[b,I_\rho]$ are bounded from $L(\Phi,\phi)(\R^n)$ to $L(\Psi,\phi)(\R^n)$ for any  $b\in \mathcal{L}_{1,\psi}(\R^n)$ and compact from $L(\Phi,\phi)(\R^n)$ to $L(\Psi,\phi)(\R^n)$ for any $b\in \overline{C_c^\infty(\R^n)}^{\|\cdot\|_{\mathcal{L}_{1,\psi}}}$ under corresponding restrictions on $\phi$, $\Phi$ and $\Psi$ . In this section, our proof will incorporate ideas and lemmas from  Nakai and Arai's works \cite{AN18, AN19, N93, N04}. Starting with establishing the boundedness of $T_{\Omega,\alpha}^m$ in Theorem \ref{T Bds}, we proceed to establish the boundedness of $[b,T_{\Omega,\alpha}^m]$ in Theorem \ref{bT Bts}.
\begin{lemma}\rm(\cite[Theorem 2.2]{RU14})\em\label{em M2}
	Let $\alpha\in [0,n)$, $ m\geqslant1$, $\{p_i,q_i,s\}^m_{i=1}$ be constants connected to $T_{\Omega,\alpha}^m$ as before and $\{A_i,\Omega_i\}$ satisfy $(\ref{A}) and (\ref{w}).$  For any $0<\delta\leqslant 1$ and $f$ such that $M^\sharp(T_{\Omega,\alpha}^m f)^\delta$ is well-defined,
	\begin{align}
		\bigg(M^\sharp |T_{\Omega,\alpha}^m f|^\delta(x)\bigg)^{\frac{1}{\delta}}\lesssim \sum\limits_{j=1}^m M_{\alpha,s}f(A_j^{-1}x).
	\end{align}
\end{lemma}
\begin{theorem}\label{T Bds}
	Let $\alpha\in [0,n)$, $ m\geqslant1$, $\{p_i,q_i,s\}^m_{i=1}$ be constants connected to $T_{\Omega,\alpha}^m$ as before. Assume that $\{A_i,\Omega_i\}$ satisfy $(\ref{A}) and (\ref{w}),$ $\Phi,\Psi \in \Delta_2\cap\nabla_2$ with some $p\in p_{\nabla_2}(\Phi)$ bigger than $s$ and  $\phi\in \mathcal{G}^{dec}$ satisfies (\ref{phi}).
	Then $T_{\Omega,\alpha}^m$ is bounded from $L(\Phi,\phi)(\R^n)$ to $L(\Psi,\phi)(\R^n)$ if
	$R^{\alpha }\Phi^{-1}\big(\phi(R)\big) \lesssim \Psi^{-1}\big(\phi(R)\big)$ for any $R>0$
	and
	\begin{align}\label{C0}
		\int_R^\infty t^\alpha\Phi^{-1}(\phi(t))\frac{dt}{t}\to \infty\ \ as\ \  R\to \infty.
	\end{align}
\end{theorem}
\begin{proof}
	Let us first consider $f\in L(\Phi,\phi)(\R^n)$ with compact support.  From  Lemma \ref{Pro} (3) we know that $L(\Phi,\phi)(\R^n)\subset L_{loc}^p(\R^n)$, thus $f\in L^p(\R^n)$ and then $T_{\Omega,\alpha}^m f \in L^{\frac{np}{n-\alpha p}}(\R^n)$ by  the $L^p(\R^n)-L^{\frac{np}{n-\alpha p}}(\R^n)$  boundedness of $T_{\Omega,\alpha}^m$ in \cite{GU99}. Hence $\lim\limits_{r\to\infty}\fint_{B(r)} |T_{\Omega,\alpha}^m f(x)|\,dx=0$ and this guarantees that we can use Lemma \ref{em M1} and Lemma \ref{em M2}.
	
	Note that there exists positive constants $C_1$ and $C_2$ such that
	\begin{align*}
		&\inf\bigg\{\lambda: \frac{1}{\phi(r)}\int_{B(x,r)} \Psi\bigg(\frac{M_{\alpha,s}f(A_j^{-1}y)}{\lambda}\bigg)\,dy\leqslant 1 \bigg\}\\
		&\leqslant C_1 \inf\bigg\{\lambda: \frac{1}{\phi(C_2r)}\int_{B(A_jx,C_2r)} \Psi\bigg(\frac{M_{\alpha,s}f(y)}{\lambda}\bigg)\,dy\leqslant 1 \bigg\}
	\end{align*}
	for any ball $B(x,r)\subset \R^n$ and this together with the results in Section 2.4 shows that
	\begin{align*}
	    \|T_{\Omega,\alpha}^m f\|_{L(\Psi,\phi)}
	    &\lesssim \|M^{\sharp}(T_{\Omega,\alpha}^m f)\|_{L(\Psi,\phi)}\\
		&=\sup\limits_{B(x,r)} \inf\bigg\{\lambda: \frac{1}{\phi(r)}\int_{B(x,r)} \Psi\bigg(\frac{M^{\sharp}(T_{\Omega,\alpha}^m f)(y)}{\lambda}\bigg)\,dy\leqslant 1 \bigg\}\\
		&\leqslant C \sup\limits_{B(x,r)}\inf\bigg\{\frac{\lambda}{C}: \frac{1}{\phi(r)}\int_{B(x,r)} \Psi\bigg(\frac{C \sum\limits_{j=1}^m M_{\alpha,s}f(A_j^{-1}y)}{\lambda}\bigg)\,dy\leqslant 1 \bigg\}
        \end{align*}
	    \begin{align*}
		&\lesssim \sum\limits_{j=1}^m \sup\limits_{B(x,r)}\inf\bigg\{\lambda: \frac{1}{\phi(r)}\int_{B(x,r)} \Psi\bigg(\frac{M_{\alpha,s}f(A_j^{-1}y)}{\lambda}\bigg)\,dy\leqslant 1 \bigg\}\\
		&\lesssim \sum\limits_{j=1}^m \sup\limits_{B(x,r)}\inf\bigg\{\lambda: \frac{1}{\phi(r)}\int_{B(x,r)} \Psi\bigg(\frac{M_{\alpha,s}f(y)}{\lambda}\bigg)\,dy\leqslant 1 \bigg\}\\
		&\lesssim \|M_{\alpha,s}f\|_{L(\Psi,\phi)}=\|M_{\alpha s}|f|^s\|_{L(\Psi(\cdot^\frac{1}{s}),\phi)}^{\frac{1}{s}}\\
		&\lesssim \||f|^s\|_{L(\Phi(\cdot^\frac{1}{s}),\phi)}^{\frac{1}{s}}=\|f\|_{L(\Phi,\phi)}.
	\end{align*}
	
	For the general case $f\in L(\Phi,\phi)(\R^n)$, choose a positive constant $M$ such that $M>2\max\limits_i\{\|A_i\|,\|A_i^{-1}\|^{-1}\}$. From Lemma \ref{Pro} (2) and what has already been proved, we obtain
	\begin{align}\label{* 1}
		\fint_{B(r)} |T_{\Omega,\alpha}^m (f\chi_{B(Mr)})(y)|\,dy
		&\nonumber\leqslant 2\Psi^{-1}(\phi(r))\|T_{\Omega,\alpha}^m (f\chi_{B(Mr)})\|_{L(\Psi,\phi)}\\
		&\lesssim \Psi^{-1}(\phi(r))\|f\|_{L(\Phi,\phi)}
	\end{align}
	for any $r>0$. Besides, the definition of $T_{\Omega,\alpha}^m$ and the assumptions stated above yield
	\begin{align}\label{***}
		&\nonumber \big|T_{\Omega,\alpha}^m (f\chi_{(B(Mr))^c})(x)\big|=\int_{(B(Mr))^c}\frac{\Omega_1(x-A_1 y)}{|x-A_1 y |^{\frac{n}{q_1}}} \cdots \frac{\Omega_m(x-A_m y)}{|x-A_m y |^{\frac{n}{q_m}}} f(y)dy\\
		&\nonumber\leqslant \sum\limits_{j=j_M}^\infty \int_{B(2^{j+1} r)\backslash B(2^j r)}\frac{\Omega_1(x-A_1 y)}{|x-A_1 y |^{\frac{n}{q_1}}} \cdots \frac{\Omega_m(x-A_m y)}{|x-A_m y |^{\frac{n}{q_m}}} f(y)dy\\
		&\nonumber\leqslant \sum\limits_{j=j_M}^\infty (2^j r)^{\alpha-n} \prod_{i=1}^m\bigg(\int_{|x-A_i y|\sim2^j r} |\Omega_i(x-A_iy)|^{p_i}\,dy\bigg)^{\frac{1}{p_i}} \bigg(\int_{B(2^{j+1} r)} |f(y)|^s\,dy\bigg)^{\frac{1}{s}}\\
		&\lesssim\sum\limits_{j=j_M}^\infty (2^j r)^{\alpha-n} \prod_{i=1}^m(2^j r)^{\frac{n}{p_i}} \bigg(\int_{B(2^{j+1} r)} |f(y)|^s\,dy\bigg)^{\frac{1}{s}}
	\end{align}
	for any $r>0$ and $x\in B(r)$,  where $j_M$ is the integer such that $2^{j_M}\leqslant M< 2^{j_M+1}$. Hence by Lemma \ref{Pro} (3) we have
	\begin{align}\label{* 2}
		\nonumber\fint_{B(r)} |T_{\Omega,\alpha}^m (f\chi_{(B(Mr))^c})(x)|\,dx
		&\lesssim \sum\limits_{j=j_M}^\infty (2^j r)^\alpha \bigg(\fint_{B(2^j r)} |f(y)|^p\,dy\bigg)^{\frac{1}{p}}\\
		&\lesssim \sum\limits_{j=j_M}^\infty (2^j r)^\alpha \Phi^{-1}\big(\phi(2^j r)\big)\|f\|_{L(\Phi,\phi)}.
	\end{align}
	
	After combining (\ref{C0}), (\ref{* 1}) and (\ref{* 2}), it may be concluded that  $$T_{\Omega,\alpha}^m f \in L_{loc}^1(\R^n)\ \ and\ \ \lim\limits_{r\to\infty}\fint_{B(r)} |T_{\Omega,\alpha}^m f(x)|dx=0.$$
	Then repeat the argument in the first case and we obtain
	\begin{align}\label{result}
		\|T_{\Omega,\alpha}^m f \|_{L(\Psi,\phi)} \lesssim \|f\|_{L(\Phi,\phi)}.
	\end{align}
\end{proof}
\begin{definition}
	$A$ is a invertible matrix. A function $f\in \mathcal{L}_{1,\psi}(\R^n)$ is said to lie in $\mathcal{L}^{A}_{1,\psi}(\R^n)$ if and only if
	there exists $C>0$ such that $$\frac{1}{\psi(r)}\bigg|\fint_{A\big(B(x,r)\big)}b-\fint_{B(x,r)}b\bigg|<C$$
	for any ball $B=B(x,r)\subset \R^n$. If $\psi=1$, then $\mathcal{L}^{A}_{1,\psi}(\R^n)=BMO^A(\R^n)$. If $A=I$, then $\mathcal{L}^{A}_{1,\psi}(\R^n)=\mathcal{L}_{1,\psi}(\R^n)$.
\end{definition}
\begin{theorem}\label{bT Bts}
	Let $\alpha\in [0,n)$,  $m\geqslant1$, $\{p_i,q_i,s\}^m_{i=1}$ be constants connected to $T_{\Omega,\alpha}^m$ as before and suppose that $\{A_i,\Omega_i\}$ satisfy $(\ref{A}) and (\ref{w}). $
	We make the extra assumptions as follows:\\
	(1) $\Phi,\Psi\in \Delta_2\cap\nabla_2$ with some $p>s$ in $p_{\nabla_2}(\Phi)$, $\phi\in \mathcal{G}^{dec}$ satisfies $(\ref{phi})$ and $\psi\in \mathcal{G}^{inc}$;\\
	(2) There exists $\Theta\in\nabla_2$ such that for any $R>0$,
	\begin{align}\label{C2}
		\nonumber R^{\alpha }\Phi^{-1}\big(\phi(R)\big) \lesssim\Theta^{-1}\big(\phi(R)\big),\\
		\nonumber \psi(R)\Theta^{-1}(\phi(R))\lesssim \Psi^{-1}(\phi(R)),\\
		R^\alpha\psi(R)\Phi^{-1}(\phi(R))\lesssim \Psi^{-1}(\phi(R));
	\end{align}
	(3) As $R\to \infty$,
	\begin{align}\label{C3}
		\int_R^\infty t^\alpha\Phi^{-1}(\phi(t))\psi(t)\log t\frac{dt}{t}\to 0 \ \ and \ \ \psi(R)\log R\cdot R^{\alpha-\frac{n}{s}}\to 0;
	\end{align}
	(4) For any $i\in\{1,2,\cdots,m\}$, there exists $\varepsilon_i>0$ such that $\int_0^1 \frac{\omega_{i,p_i}(t)}{t^{1+\varepsilon_i}}dt<\infty.$\\
	
	Then $[b,T_{\Omega,\alpha}^m]$ is well-defined for any function $f\in L(\Phi,\phi)(\R^n)$ and  $b\in\cap_{i=1}^m\mathcal{L}^{A_i}_{1,\psi}(\R^n)$. Moreover, $[b,T_{\Omega,\alpha}^m]$ is bounded from $L(\Phi,\phi)(\R^n)$ to $L(\Psi,\phi)(\R^n).$
\end{theorem}

\begin{proof}
	Our proof starts with the boundedness of some related operators. By \cite[Lemma 4.4]{SAN19} we can choose a constant $\eta>1$ such that $\Phi(\cdot^{\frac{1}{\eta}})$, $\Theta(\cdot^{\frac{1}{\eta}})$ and $\Psi(\cdot^{\frac{1}{\eta}})$ are Young functions satisfying $\nabla_2$-condition. From  Lemma \ref{M Bds}, Theorem \ref{T Bds} and (\ref{C2}), we see that
	$T_{\Omega,\alpha}^m$ is bounded from $L(\Phi,\phi)(\R^n)$ to $L(\Theta,\phi)(\R^n)$, $M_{\psi^\eta}$ is bounded from $L\big(\Theta(\cdot^{\frac{1}{\eta}}),\phi\big)(\R^n)$ to $L\big(\Psi(\cdot^{\frac{1}{\eta}}),\phi\big)(\R^n)$ and $M_{\big((\cdot)^\alpha \psi(\cdot)\big)^\eta}$ is bounded from $L\big(\Phi(\cdot^{\frac{1}{\eta}}),\phi\big)(\R^n)$ to $L\big(\Psi(\cdot^{\frac{1}{\eta}}),\phi\big)(\R^n)$. \\
	
	Claim: For any $x\in \R^n$,  $b\in\cap_{i=1}^m\mathcal{L}^{A_i}_{1,\psi}(\R^n)$ and $f\in L(\Phi,\phi)$,
	\begin{align*}
		M^\sharp\big([b,T_{\Omega,\alpha}^m]f\big)(x)
		&\lesssim \sum\limits_{i=1}^m \big(C_i(b)+\|b\|_{\mathcal{L}_{1,\psi}}\big)\bigg[\bigg(M_{\psi^\eta}\big|T_{\Omega,\alpha}^m f\big|^\eta(x)\bigg)^{\frac{1}{\eta}}\\
		&\ \ \ \ +\bigg(M_{\big((\cdot)^\alpha \psi(\cdot)\big)^\eta}|f\circ A_i^{-1}|^\eta(x)\bigg)^{\frac{1}{\eta}}
		\bigg],
	\end{align*}
	where for any $i\in \{1,2,\cdots, m\}$, $C_i(b)$ denotes one of the constants satisfying $$\sup\limits_{B=B(x,r)}\frac{1}{\psi(r))}|\fint_{B(A_ix,r)}b-\fint_{B(x,r)}b|\leqslant C_i(b).$$
	proof of Claim: Given $x\in \R^n$, set  $C(x)=T_{\Omega,\alpha}^m\big((b-\fint_{MB}b)f\chi_{(MB)^c}\big)(x)$ and $MB=B(x,Mr)$ with $M>2\max\limits_i\{\|A_i\|,\|A_i^{-1}\|^{-1}\}$. Then
	\begin{align*}
		M^\sharp\big([b,T_{\Omega,\alpha}^m]f\big)(x)
		&=\sup\limits_{B=B(x,r)} \fint_{B} \big|[b,T_{\Omega,\alpha}^m]f(y)-\fint_B [b,T_{\Omega,\alpha}^m]f(z)dz\big|dy\\
		&=\sup\limits_{B=B(x,r)} \fint_{B} \bigg|[b,T_{\Omega,\alpha}^m]f(y)-C(x)-\fint_B \bigg([b,T_{\Omega,\alpha}^m]f(z)-C(x)\bigg)dz\bigg|dy\\
		&\leqslant 2\sup\limits_{B=B(x,r)} \fint_{B} \big|[b,T_{\Omega,\alpha}^m]f(y)-C(x)\big|dy\\
		&=2\sup\limits_{B=B(x,r)} \fint_{B} \big|[b-\fint_{MB} b,T_{\Omega,\alpha}^m]f(y)-C(x)\big|dy\\
		&\lesssim \sup\limits_{B=B(x,r)}\fint_{B}\big|\big(b(y)-\fint_{MB} b\big)T_{\Omega,\alpha}^m f(y)\big|dy\\
		&+\sup\limits_{B=B(x,r)}\fint_{B}\big|T_{\Omega,\alpha}^m \big(b-\fint_{MB} b\big)(f\chi_{MB})(y)\big|dy\\
		&+\sup\limits_{B=B(x,r)}\fint_{B}\big|T_{\Omega,\alpha}^m \big(b-\fint_{MB} b\big)(f\chi_{{(MB)}^c})(y)-C(x)\big|dy.
	\end{align*}
	Now we are reduced to  deal with the last three items above respectively.
	
	According to \cite[Corollary 4.3]{AN18}, the first item is controlled by
	\begin{align*}
		&\sup\limits_{B=B(x,r)}\bigg(\fint_B\big|b(y)-\fint_{MB} b\big|^{\eta'}dy\bigg)^{\frac{1}{\eta'}} \bigg(\fint_B\big|T_{\Omega,\alpha}^m(y)dy\big|^{\eta}\bigg)^{\frac{1}{\eta}}\\
		&\leqslant \sup\limits_{B=B(x,r)}\bigg(\frac{|MB|}{|B|[\psi(B)]^{\eta'}}\fint_{MB}\big|b(y)-\fint_{MB} b\big|^{\eta'}dy\bigg)^{\frac{1}{\eta'}} \bigg([\psi(B)]^{\eta}\fint_B\big|T_{\Omega,\alpha}^m(y)dy\big|^{\eta}\bigg)^{\frac{1}{\eta}}\\
		&\lesssim \|b\|_{\mathcal{L}_{1,\psi}}\bigg(M_{\psi^\eta}\big|T_{\Omega,\alpha}^m f\big|^\eta(x)\bigg)^{\frac{1}{\eta}}.
	\end{align*}
	
	For the second item, notice that $(b-\fint_{MB} b)f\chi_{MB}\in L^{\tilde p}(\R^n)$ for some $\tilde p>1$. Hence
	$T_{\Omega,\alpha}^m \big(b-\fint_{MB} b\big)(f\chi_{MB})\in L^{\tilde q}(\R^n)$ when $\frac{1}{\tilde q}=\frac{1}{\tilde p}-\frac{\alpha}{n}$ and consequently
	\begin{align*}
		&\sup\limits_{B=B(x,r)}\fint_{B}\big|T_{\Omega,\alpha}^m \big(b-\fint_{MB} b\big)(f\chi_{MB})(y)\big|dy\\
		&\leqslant\sup\limits_{B=B(x,r)}\bigg(\fint_{B}\big|T_{\Omega,\alpha}^m \big(b-\fint_{MB} b\big)(f\chi_{MB})(y)\big|^{\tilde q}dy\bigg)^{\frac{1}{\tilde q}}
		\end{align*}
		\begin{align*}
		&\leqslant \sup\limits_{B=B(x,r)}\frac{1}{|B|^{{\frac{1}{\tilde q}}}}\big\| T_{\Omega,\alpha}^m \big(b-\fint_{MB} b\big)(f\chi_{MB})\big\|_{L^{\tilde q}}\\
		&\lesssim \sup\limits_{B=B(x,r)}\frac{1}{|B|^{{\frac{1}{\tilde q}}}}\big\| \big(b-\fint_{MB} b\big)(f\chi_{MB})\big\|_{L^{\tilde p}}\\
		&=\sup\limits_{B=B(x,r)}|B|^{\frac{\alpha}{n}}\cdot\frac{1}{|B|^\frac{1}{\tilde p}}\big\| \big(b-\fint_{MB} b\big)(f\chi_{MB})\big\|_{L^{\tilde p}}\\
		&\lesssim \sup\limits_{B=B(x,r)} r^\alpha \bigg(\fint_{MB} \big|(b-\fint_{MB} b)f(y)\big|^{\tilde p}\bigg)^{\frac{1}{\tilde p}}\\
		&\leqslant \sup\limits_{B=B(x,r)} r^\alpha \cdot \bigg(\frac{1}{{\psi(r)}^{\eta_1}}\fint_{MB}|b(y)-\fint_{MB} b|^{\eta_1}dy\bigg)^{\frac{1}{\eta_1}}\bigg({\psi(r)}^{\eta}\fint_{MB}|f(y)|^{\eta}dy\bigg)^{\frac{1}{\eta}}\\
		&\lesssim \|b\|_{\mathcal{L}_{1,\psi}}\bigg(M_{\big((\cdot)^\alpha \psi(\cdot)\big)^\eta}|f|^\eta(x)\bigg)^{\frac{1}{\eta}}
	\end{align*}
	where $\frac{1}{\tilde p}=\frac{1}{\eta}+\frac{1}{\eta_1}$. Indeed, the third item is
	\begin{align*}
		&\sup\limits_{B}
		\fint_{B}\bigg|
		\int_{{(MB)}^c} (b(z)-\fint_{MB} b)K(y,z)f(z)dz-\int_{{(MB)}^c} (b(z)-\fint_{MB} b)K(x,z)f(z)dz
		\bigg|dy \\
		&\leqslant \sup\limits_{B=B(x,r)}
		\fint_{B}
		\int_{{(MB)}^c} |b(z)-\fint_{MB} b|\big|K(y,z)-K(x,z)\big||f(z)|dz
		dy.
	\end{align*}
	Take advantage of some estimates in \cite{RU17} and we conclude  that for any $y\in B$ and $z\in {(MB)}^c$,
	\begin{align*}
		&\int_{{(MB)}^c} |b(z)-\fint_{MB} b|\big|K(y,z)-K(x,z)\big||f(z)|dz\\
		&\lesssim\sum\limits_{i=1}^m
		\sum\limits_{j=0}^\infty\bigg(\int_{A^i_j} |(b(z)-\fint_{MB}b)f(z)|^s\,dz\bigg)^{\frac{1}{s}}\cdot(2^{j}Mr)^{\alpha-\frac{n}{s}}\\
		&\ \ \ \ \ \ \ \ \ \ \ \ \cdot \bigg[\frac{
			|y-x|
		}{2^{j}Mr}+\int_{\frac{
				|y-x|
			}{2^{j}Mr}}^{\frac{
				|y-x|
			}{2^{j+1}Mr}}\frac{\omega_{1,p_1}(\delta)}{\delta}d\delta\bigg]\\
		&\lesssim\sum\limits_{i=1}^m
		\sum\limits_{j=0}^\infty\bigg(\fint_{A^i_j} |(b(z)-\fint_{MB}b)|^{\eta_2}\,dz\bigg)^{\frac{1}{\eta_2}}
		\bigg(\fint_{A^i_j}|f(z)|^{\eta}dz\bigg)^{\frac{1}{\eta}}\cdot(2^{j}Mr)^{\alpha}\\
		&\ \ \ \ \ \ \ \ \ \ \ \ \cdot\bigg[\frac{
			1
		}{2^{j}M}+\int_{\frac{
				|y-x|
			}{2^{j}Mr}}^{\frac{
				|y-x|
			}{2^{j+1}Mr}}\frac{\omega_{1,p_1}(\delta)}{\delta}d\delta\bigg]\\
	\end{align*}
	where $A^i_j=\{z: 2^{j}Mr\leqslant|x-A_i x_a|\leqslant2^{j+1}Mr\}$ and $\frac{1}{s}=\frac{1}{\eta}+\frac{1}{\eta_2}.$
	Notice that
	$$\bigg(\big((2^{j}Mr)^{\alpha\eta}\cdot\psi(2^{j+1}Mr)\big)^{\eta}\fint_{A^i_j}|f(z)|^{\eta}dz\bigg)^{\frac{1}{\eta}}\lesssim \bigg(M_{\big((\cdot)^\alpha \psi(\cdot)\big)^\eta}|f\circ A_i^{-1}|^\eta(x)\bigg)^{\frac{1}{\eta}}$$
	and
	\begin{align*}
		&\bigg(\frac{1}{[\psi(2^{j+1}Mr)]^{\eta_2}}\fint_{A^i_j}|(b(z)-\fint_{MB}b)|^{\eta_2}\,dz\bigg)^{\frac{1}{\eta_2}}\\
		&\lesssim\bigg(\frac{1}{[\psi(2^{j+1}Mr)]^{\eta_2}}\fint_{B(A_i^{-1}x,2^{j+1}Mr)}|(b(z)-\fint_{B(A_i^{-1}x,Mr)}b)|^{\eta_2}\,dz\bigg)^{\frac{1}{\eta_2}}\\
		&+\frac{1}{\psi(2^{j+1}Mr)}|\fint_{B(A_i^{-1}x,Mr)}b-\fint_{B(x,Mr)}b|\\
		&\lesssim \|b\|_{\mathcal{L}_{1,\psi}}+\sup\limits_{B=B(x,r)}\frac{1}{\psi(Mr)}|\fint_{B(A_i^{-1}x,Mr)}b-\fint_{B(x,Mr)}b|
	\end{align*}
	by  $\psi \in\mathcal{G}^{inc}$and \cite[Remark 4.1]{AN18}. Therefore the third item is controlled by
	\begin{align*}
		&\sum\limits_{i=1}^m
		\sum\limits_{j=0}^\infty (j+1)\bigg[\frac{
			1
		}{2^{j}M}+\int_{\frac{
				|y-x|
			}{2^{j}Mr}}^{\frac{
				|y-x|
			}{2^{j+1}Mr}}\frac{\omega_{1,p_1}(\delta)}{\delta}d\delta\bigg] \cdot \big(C_i(b)+\|b\|_{\mathcal{L}_{1,\psi}}\big)\\
		&\ \ \ \ \ \ \ \ \cdot\bigg(M_{\big((\cdot)^\alpha \psi(\cdot)\big)^\eta}|f\circ A_i^{-1}|^\eta(x)\bigg)^{\frac{1}{\eta}}\\
		&\lesssim
		\sum\limits_{i=1}^m \big(C_i(b)+\|b\|_{\mathcal{L}_{1,\psi}}\big)\bigg(M_{\big((\cdot)^\alpha \psi(\cdot)\big)^\eta}|f\circ A_i^{-1}|^\eta(x)\bigg)^{\frac{1}{\eta}}
	\end{align*}
	in case of the assumption (4) and  the definition of $\mathcal{L}_{1,\psi}^A(\R^n)$. Now the proof of the Claim is completed.
	
	As in the proof of Theorem \ref{T Bds}, we first consider $f\in L(\Phi,\phi)(\R^n)$ with support in $B(R)$.
	By the Claim and Lemma \ref{em M1}, we only need to show that $\lim\limits_{r\to\infty}\fint_{B(0,r)} \big|[b,T_{\Omega,\alpha}^m] f(x)\big|\,dx=0$.
	
	Since $f \in L^p(\R^n)$ with $supp f\subset B(R)$ and $b\in L^{p_0}_{loc}(\R^n)$ for any $p_0>1$, we deduce that  $T_{\Omega,\alpha}^m f\in L^{\frac{np}{n-\alpha p}}(\R^n)$ and $b\cdot T_{\Omega,\alpha}^m f\cdot \chi_{B(MR)}\in L^1(\R^n)$. Hence we have
	\begin{align*}
		S_1(r):=\fint_{B(r)} \big|b(x)\cdot T_{\Omega,\alpha}^m f(x)\cdot \chi_{B(MR)}(x)\big|dx\to0\ \ as\ \ r\to\infty,
	\end{align*}
	
	\begin{align*}
		S_2(r):=\fint_{B(r)} \big|T_{\Omega,\alpha}^m (bf)(x)\cdot \chi_{B(MR)}(x)\big|dx\to0\ \ as\ \ r\to\infty,
	\end{align*}
	\begin{align*}
		S_3(r):=\fint_{B(r)} \big|T_{\Omega,\alpha}^m (bf)(x)\cdot \chi_{(B({MR}))^c}(x)\big|dx\to0\ \ as\ \ r\to\infty
	\end{align*}
	and
	\begin{align*}
		S_4(r):=\fint_{B(r)} \big|\big(\fint_{B(MR)} b \big)\cdot T_{\Omega,\alpha}^m f(x)\cdot \chi_{(B({MR}))^c}(x)\big|dx\to0\ \ as\ \ r\to\infty.
	\end{align*}
	Besides,
	\begin{align*}
		S_5(r)&:=\fint_{B(r)} \big|\big(b(x)-\fint_{B(MR)} b \big)\cdot T_{\Omega,\alpha}^m f(x)\cdot \chi_{(B({MR}))^c}(x)\big|dx\\
		&\leqslant\bigg(\fint_{B(r)}\big|b(x)-\fint_{B(MR)} b \big|^{\v'}\bigg)^{\frac{1}{\v'}}\bigg(\fint_{B(r)}\big|T_{\Omega,\alpha}^m f(x)\cdot \chi_{(B({MR}))^c}(x)\big|^{\v}\bigg)^{\frac{1}{\v}}
	\end{align*}
	for some $\v>1$. From \cite[Remark 4.1]{AN18} we see that
	\begin{align*}
		\bigg(\fint_{B(r)}\big|b(x)-\fint_{B(MR)} b \big|^{\v'}\bigg)^{\frac{1}{\v'}}\lesssim \psi(r)\log \frac{r}{MR}\cdot \|b\|_{\mathcal{L}_{1,\psi}}.
	\end{align*}
	On the other hand,
	\begin{align*}
		&\bigg(\fint_{B(r)}\big|T_{\Omega,\alpha}^m f(x)\cdot \chi_{(B({MR}))^c}(x)\big|^{\v}\bigg)^{\frac{1}{\v}}\\
		&\leqslant
		r^{-\frac{n}{\v}}\sum\limits_{j=0}^{j_r} \|T_{\Omega,\alpha}^m f(x)\chi_{|x|\sim 2^j MR}\|_{L^\v(\R^n)}\\
		&\lesssim r^{-\frac{n}{\v}}\sum\limits_{j=0}^{j_r} (2^j MR)^{\alpha-\frac{n}{s}}\cdot \|f\|_{L^s(\R^n)}\cdot |B(2^j MR)|^{\frac{1}{\v}}\\
		&\lesssim \sum\limits_{j=0}^{j_r} (2^j MR)^{\alpha-\frac{n}{s}}\cdot \|f\|_{L^s(\R^n)}\leqslant \|f\|_{L^s(\R^n)}\cdot r^{\alpha-\frac{n}{s}},
	\end{align*}
	where $j_r$ is the smallest integer such that $r\leqslant2^{j_r} MR$.
	Hence (\ref{C3}) yields
	$$S_5(r)\lesssim \|b\|_{\mathcal{L}_{1,\psi}} \|f\|_{L^s(\R^n)}\cdot\log r\cdot \psi(r)\cdot
	r^{\alpha-\frac{n}{s}}\to 0\ \ as\ \ r\to\infty.$$
	From the above argument we obtain
	$$\lim\limits_{r\to\infty}\fint_{B(0,r)} \big|[b,T_{\Omega,\alpha}^m] f(x)\big|\,dx=\lim\limits_{r\to\infty} \sum\limits_{i=1}^5 S_i(r)=0.$$
	
	For the general case, we fix $b\in\cap_{i=1}^m\mathcal{L}^{A_i}_{1,\psi}(\R^n)$,  $f\in L(\Phi,\phi)(\R^n)$ and note that  $$\|f\circ A_i^{-1}\|_{L(\Phi,\phi)}\lesssim\|f\|_{L(\Phi,\phi)}\ \ for\ any\ i\in\{1,2,\cdots,m\}.$$
	From Lemma \ref{Pro} (2) and what has been proved,  we have
	\begin{align}\label{** 1}
		\nonumber\fint_{B(r)} \big|[b,T_{\Omega,\alpha}^m](f\chi_{B(Mr)})(x)\big|\,dx
		&\leqslant 2\Psi^{-1}(\phi(r))\big\|[b,T_{\Omega,\alpha}^m] (f\chi_{B(Mr)})\big\|_{L(\Psi,\phi)}\\
		&\lesssim \Psi^{-1}(\phi(r))\sum\limits_{i=1}^m \big(C_i(b)+\|b\|_{\mathcal{L}_{1,\psi}}\big)\|f\|_{L(\Phi,\phi)}.
	\end{align}
	Besides, we will show that
	\begin{align}\label{** 2}
		&\nonumber\fint_{B(r)} \big|[b,T_{\Omega,\alpha}^m](f\chi_{(B(Mr))^c})(x)\big|\,dx\\
		&\nonumber\lesssim \|b\|_{\mathcal{L}_{1,\psi}}\|f\|_{L(\Phi,\phi)}\bigg(\int_{Mr}^\infty t^\alpha\Phi^{-1}\big(\phi(t)\big)\frac{dt}{t}\cdot \psi(r) \bigg)\\   &+\|b\|_{\mathcal{L}_{1,\psi}}\|f\|_{L(\Phi,\phi)}\bigg(\int_{Mr}^\infty t^\alpha\Phi^{-1}\big(\phi(t)\big)\psi(t)\log t\frac{dt}{t} \bigg).
	\end{align}
	To obtain (\ref{** 2}), notice that
	\begin{align*}
		&\fint_{B(r)} \big|[b,T_{\Omega,\alpha}^m](f\chi_{(B(Mr))^c})(x)\big|\,dx\\
		&\leqslant \fint_{B(r)} \big|\int_{(B(Mr))^c} \big(b(x)-b(y)\big) K(x,y) f(y)\,dy\big|\,dx\\
		&\leqslant \fint_{B(r)} \big|b(x)-\fint_{B(r)} b\big|
		\bigg|\int_{(B(Mr))^c)}  K(x,y) f(y)\,dy\bigg|\,dx\\
		&+\fint_{B(r)}
		\int_{(B(Mr))^c}  \big|b(y)-\fint_{B(r)} b\big| |K(x,y)f(y)|\,dy\,dx.
	\end{align*}
	While (\ref{***}) and (\ref{* 2}) imply
	\begin{align*}
		&\fint_{B(r)} \big|b(x)-\fint_{B(r)} b\big|\bigg|\int_{(B(Mr))^c}  K(x,y) f(y)\,dy\bigg|\,dx\\
		&\leqslant \|b\|_{\mathcal{L}_{1,\psi}} \cdot\psi(r)\cdot\sup\limits_{x\in B(r)} T_{\Omega,\alpha}^m (f\chi_{(B(Mr))^c})(x)\\
		&\lesssim \|b\|_{\mathcal{L}_{1,\psi}}\|f\|_{L(\Phi,\phi)}\bigg(\int_{Mr}^\infty t^\alpha\Phi^{-1}\big(\phi(t)\big)\frac{dt}{t}\cdot \psi(r) \bigg).
	\end{align*}
	Furthermore for any $x\in B(r)$, by similar argument of (\ref{***}) and (\ref{* 2}), we have
	\begin{align*}
		&\int_{(B(Mr))^c}  \big|b(y)-\fint_{B(r)} b\big| |K(x,y)f(y)|\,dy\\
		&\lesssim \sum\limits_{j=j_M}^\infty (2^j r)^\alpha \bigg(\fint_{B(2^{j+1} r)} |f(y)|^{s_1}\,dy\bigg)^{\frac{1}{s_1}}\bigg(\fint_{B(2^{j+1} r)} \big|b(y)-\fint_{B(r) }b\big|^{s_2}\,dy\bigg)^{\frac{1}{s_2}}\\
		&\lesssim \sum\limits_{j=j_M}^\infty (2^j r)^\alpha \Phi^{-1}\big(\phi(2^j r)\big)\|f\|_{L(\Phi,\phi)}\cdot \psi(2^j r)\cdot j\cdot \|b\|_{\mathcal{L}_{1,\psi}}\\
		&\leqslant \|b\|_{\mathcal{L}_{1,\psi}}\|f\|_{L(\Phi,\phi)}\bigg(\int_{Mr}^\infty t^\alpha\Phi^{-1}\big(\phi(t)\big)\psi(t)\log t\frac{dt}{t} \bigg),
	\end{align*}
	where $s_1<p$, $\frac{1}{s_1}+\frac{1}{s_2}=\frac{1}{s}$ and $j_M$ is the integer such that $2^{j_M}\leqslant M<2^{j_M+1}$.
	After combining (\ref{C3}), (\ref{** 1}) and (\ref{** 2}), it follows that $$\lim\limits_{r\to\infty}\fint_{B(0,r)} \big|[b,T_{\Omega,\alpha}^m] f(x)\big|\,dx=0$$
	and then still by the Claim and Lemma \ref{em M1} we finally obtain
	\begin{align}\label{result 2}
		\|[b,T_{\Omega,\alpha}^m] f \|_{L(\Psi,\phi)} \lesssim \sum\limits_{i=1}^m \big(C_i(b)+\|b\|_{\mathcal{L}_{1,\psi}}\big)\|f\|_{L(\Phi,\phi)}
	\end{align}
	for any $b\in\cap_{i=1}^m\mathcal{L}^{A_i}_{1,\psi}(\R^n)$ and $f\in L(\Phi,\phi)(\R^n)$.\\
\end{proof}

\begin{example}
	Let $s<p\leqslant q =\frac{np}{n-\alpha p}$.
	Set $\phi(t)=t^{-n}$, $\psi(t)=1$, $\Phi(t)=t^p$, and $\Theta(t)=\Psi(t)=t^q$. By Theorem \ref{T Bds} and Theorem \ref{bT Bts}, we can conclude that $T_{\Omega,\alpha}^m$ is bounded from $L^p(\R^n)$ to $L^q(\R^n)$ and $[b,T_{\Omega,\alpha}^m]$ is bounded from $L^p(\R^n)$ to $L^q(\R^n)$ for any $b\in \cap_{i=1}^m BMO^{A_i}(\R^n)$.
\end{example}


\section{Compactness of $[b, \tilde T_{\Omega,\alpha}^1]$ from $L(\Phi,\varphi)(\R^n)$ to $L(\Psi,\varphi)(\R^n)$}

In this paper, no attempt has been made to develop the compactness of the operator whose kernel has more than one singularity since the estimates for such kernel related to two points $x,y\in \R^n$ can not be controlled as the distance between $x$ and $y$  converges to $0$. Thus recall the definition \begin{align}
	\tilde T_{\Omega,\alpha}^1 f(x)=\int_{\R^n}\frac{\Omega(x- y)}{|x- y|^{n-\alpha}} f(y)dy,
\end{align}
where $\Omega$ is homogeneous of degree 0 on $\R^n$ and $\Omega\in L^{r}(\Sigma)$ for some $r\geqslant1$.  In \cite{CT15}, the authors considered the compactness of the  commutaters for the bilinear fractional integral operator with $CMO$ functions. The compactness results for the commutaters of $\tilde T_{\Omega,\alpha}^1$ with $CMO$ functions on weighted spaces are included in \cite{GHWY22, GWY21, WY18} where the authors showed that
$[b,\tilde T_{\Omega,\alpha}^1]$  is compact from $L^p_{\omega^p}(\R^n)$ to $L^q_{\omega^q}(\R^n)$ for any $\omega^s\in \mathcal{A}(\frac{p}{s}.\frac{q}{s})$ and $b\in CMO(\R^n)$. The compactness results for the commutators of $I_\rho$ with the generlized $Companato$ function are discussed in \cite{YN22}. Our focus in this section is on the compactness of $[b,\tilde T_{\Omega,\alpha}^1]$ from $L(\Phi,\varphi)(\R^n)\  to \ L(\Psi,\varphi)(\R^n)$ and here we recall two essential lemmas coming from \cite{AN19, KA82,SS08,YN22}.

\begin{lemma}\label{Pro 2}\rm(\cite[Section 6]{YN22})\em\\
	(1) If $\Phi\in\nabla_2$,  $\Theta\in \Delta_2$ and $Tf(x)=\int K(x,y)f(y)\,dy$ for some kernel $K: \R^n \times \R^n \to \C$, then $T$ is bounded and compact from $L^\Phi(\R^n)$ to $L^\Theta(\R^n)$ when
	$$\bigg\| \|K(x,y)\|_{L^{\tilde\Phi}_y} \bigg\|_{L^\Theta_x}<\infty;$$
	(2) If $\Phi\in\Delta_2$ and $\phi\in \mathcal{G}^{dec}$,
	then the operator $T: f\to f\chi_{B(x,R)}$ is bounded from $L(\Phi,\phi)(\R^n)$ to $L^\Phi(\R^n)$ for any ball $B(x,R)\subset
	\R^n$;\\
	(3) Let $\Psi\in\Delta_2$ and $\phi\in \mathcal{G}^{dec}$. Assume that $\phi$ satisfies $(\ref{phi}),$ then there exists $\Theta\in \Delta_2$ such that the operator $T: f\to f\chi_{B(x,R)}$ is bounded from $L^\Theta(\R^n)$ to $L(\Psi,\phi)(\R^n)$ for any ball $B(x,R)\subset
	\R^n$.
\end{lemma}
\begin{lemma}\label{estimate b}\rm(\cite[Lemma 5.3]{AN19})\em
	 If $b\in C^\infty_c(\R^n)$ and  $\int_R^\infty \frac{\psi(t)}{t^2}dt\lesssim \frac{\psi(R)}{R}$ for any $R>0$, then there exists $0<\theta<1$ such that
	$$|b(x)-b(y)|\lesssim \|\nabla b\|_{L^\infty}|x-y|^\theta \psi(|x-y|)$$
	for any $x,y\in\R^n$ with $|x-y|<1$.
\end{lemma}

\begin{theorem}\label{compact 2}
	Let $0\leqslant\alpha<n$, $1<s<\infty$. Suppose that $\Omega$ is homogeneous of degree 0 on $\R^n$ and $\Omega \in L^{s'}(S^{n-1})$ . Under the assumptions in Theorem \ref{bT Bts}, $[b,\tilde T_{\Omega,\alpha}^1]$ is compact from $L(\Phi,\phi)(\R^n)$ to $L(\Psi,\phi)(\R^n)$ for $b\in \overline{C_c^\infty(\R^n)}^{\|\cdot\|_{\mathcal{L}_{1,\psi}}}$ if
	\begin{align}\label{assumption}
		\int_R^\infty \frac{\psi(t)}{t^2}dt\lesssim \frac{\psi(R)}{R}\ and\ \int_R^\infty t^{\alpha-\frac{n}{s}}\frac{1}{\Psi^{-1}(\phi(t))}\frac{dt}{t}<\infty
	\end{align}
	for any $R>0$.
\end{theorem}
\begin{proof}
	For any $b\in \overline{C_c^\infty(\R^n)}^{\|\cdot\|_{\mathcal{L}_{1,\psi}}}$, by Theorem \ref{bT Bts}  there exists  functions $\{b_k\}_{k=1}^\infty\subset C^\infty_c(\R^n)$ such that
	$$\|[b,\tilde T_{\Omega,\alpha}^1]-[b_k,\tilde T_{\Omega,\alpha}^1]\|_{L(\Phi,\phi)(\R^n)\to L(\Psi,\phi)(\R^n)}\lesssim \|b-b_k\|_{\mathcal{L}_{1,\psi}}\to 0\ as\ k\to \infty.$$
	Hence it is sufficient to consider the condition $b\in C^\infty_c(\R^n)$.  Along the method in \cite{AN19,YN22}, set
	$$T^R_\varepsilon f(x)=\int_{\R^n}\frac{\Omega\chi_{B(R)\backslash B(\varepsilon)}(x-y)}{|x-y|^{n-\alpha}}f(y)\,dy$$ and
	$$T_\varepsilon f(x)=\int_{\R^n}\frac{\Omega\chi_{(B(\varepsilon))^c}(x-y)}{|x-y|^{n-\alpha}}f(y)\,dy.$$
	
	Step 1. $[b,T^R_\varepsilon]$ is compact from $L(\Phi,\phi)(\R^n)$ to  $L(\Psi,\phi)(\R^n)$.\\
	
	Since $b$ is supported in some ball $B_b$,  we can choose a ball $B^R_\varepsilon$ such that the kernel of $[b,T^R_\varepsilon]$,
	$$K^R_\varepsilon(x,y)=\big(b(x)-b(y)\big)\frac{\Omega\chi_{B(R)\backslash B(\varepsilon)}(x-y)}{|x-y|^{n-\alpha}},$$
	is supported in $B^R_\varepsilon\times B^R_\varepsilon\subset \R^n\times \R^n$ and thus
	$$|K^R_\varepsilon(x,y)|\lesssim  \chi_{B^R_\varepsilon}(x)\chi_{B^R_\varepsilon}(y)\frac{|\Omega\chi_{B(R)\backslash B(\varepsilon)}(x-y)|}{|x-y|^{n-\alpha}}.$$
	
	From Remark \ref{R2}, there exists $p_1, p_2>1$ such that $t^{p_1}\lesssim \Phi(t) $ for any $t>1$ and $t^{p_2}\lesssim \Phi(t) $ for any $0<t\leqslant1$. Hence we can find $C>0$ such that  $\tilde \Phi(t)\leqslant C t^{p_1'}$ for any $t>1$ and $\tilde \Phi(t)\leqslant C t^{p_2'}$ for any $0<t\leqslant1$. It follows that
	\begin{align*}
		\|K^R_\varepsilon(x,y)\|_{L^{\tilde\Phi}_y}
		&=\inf \bigg\{\lambda>0, \  \int_{\R^n} \tilde\Phi(\frac{|K^R_\varepsilon(x,y)|}{\lambda})\,dy\leqslant 1\bigg\}\\
		&\leqslant \inf \bigg\{\lambda>0, \  \int_{\R^n} \max\bigg\{\frac{C|K^R_\varepsilon(x,y)|^{p_1'}}{\lambda^{p_1'}},\frac{C|K^R_\varepsilon(x,y)|^{p_2'}}{\lambda^{p_2}}\bigg\}\,dy\leqslant 1\bigg\}\\
		&=\inf \bigg\{\lambda>0, \  \int_{\frac{|K^R_\varepsilon(x,y)|^{p_1'}}{\lambda^{p_1'}}\geqslant\frac{|K^R_\varepsilon(x,y)|^{p_2'}}{\lambda^{p_2'}}} \frac{C|K^R_\varepsilon(x,y)|^{p_1'}}{\lambda^{p_1'}}\,dy\\
		&\ \ \ \ \ \ \ \ \ \ \ \ \ \ \ \ \ \ \  +\int_{\frac{|K^R_\varepsilon(x,y)|^{p_1'}}{\lambda^{p_1'}}\leqslant\frac{|K^R_\varepsilon(x,y)|^{p_2'}}{\lambda^{p_2'}}}\frac{C|K^R_\varepsilon(x,y)|^{p_2'}}{\lambda^{p_2'}}\,dy\leqslant 1\bigg\}\\
		&\leqslant  (2C)^{\frac{1}{p_1'}}\|K^R_\varepsilon(x,y)\|_{L^{p_1'}_y}+(2C)^{\frac{1}{p_2'}}\|K^R_\varepsilon(x,y)\|_{L^{p_2'}_y}\\
		&\lesssim \chi_{B^R_\varepsilon}(x)
	\end{align*}
	and thus by Lemma \ref{Pro},
	\begin{align*}
		\bigg\| \|K^R_\varepsilon(x,y)\|_{L^{\tilde\Phi}_y} \bigg\|_{L^\Theta_x}\lesssim \|\chi_{B^R_\varepsilon}\|_{L^\Theta} \lesssim \frac{1}{\Theta^{-1}(\frac{2^n}{[{diam ({B^R_\varepsilon})}]^n})} <\infty.
	\end{align*}
	Combining this with Lemma \ref{Pro 2},
	$[b,T^R_\varepsilon]$ is compact from $L(\Phi,\phi)(\R^n)$ to  $L(\Psi,\phi)(\R^n)$ since
	$$T^R_\varepsilon f(x)=\chi_{B^R_\varepsilon}(x)T^R_\varepsilon (\chi_{B^R_\varepsilon}f)(x).$$
	
	Step 2. $$\|[b, T^R_\varepsilon]-[b,T_\varepsilon]\|_{L(\Phi,\phi)(\R^n)\to L(\Psi,\phi)(\R^n)}\to 0\ as\ R\to \infty.$$
	Given $f\in L(\Phi,\phi)$,
	 \begin{align}\label{*0}
		&\nonumber\big|[b, T^R_\varepsilon]f(x)-[b,T_\varepsilon]f(x)\big|\\
		&\nonumber \leqslant \int_{\R^n} |b(x)-b(y)|\frac{|\Omega\chi_{(B(R))^c}(x-y)|}{|x-y|^{n-\alpha}}|f(y)|\, dy\\
		&\lesssim\chi_{B_b}(x)\int_{\R^n} \frac{|\Omega\chi_{(B(R))^c}(x-y)|}{|x-y|^{n-\alpha}}|f(y)|\, dy+\int_{\R^n} \frac{|\Omega\chi_{(B(R))^c}(x-y)|}{|x-y|^{n-\alpha}}|f\chi_{B_b}(y)|\, dy.
	\end{align}
	
	For $x\in B_b$ and $R$ big  enough, $\{y:\ |x-y|>R\}\subset \big(B(2^{M(R)}r_b)\big)^c$ for  some $M(R)$ where $r_b=\frac{diam B_b}{2}$ and $M(R)\to \infty\ \ as\ \ R\to \infty.$
	It follows that
	\begin{align*}
		&\chi_{B_b}(x)\int_{\R^n} \frac{|\Omega\chi_{(B(R))^c}(x-y)|}{|x-y|^{n-\alpha}}|f(y)|\, dy\\
		&\leqslant \int_{(B(2^{M(R)}r_b))^c} \frac{|\Omega(x-y)|}{|x-y|^{n-\alpha}}|f(y)|\, dy
		\end{align*}
	\begin{align*}
		&\leqslant\sum\limits_{j=M(R)}^\infty \int_{B{(2^{j+1}r_b)}\backslash B{(2^{j}r_b)}} \frac{|\Omega(x-y)|}{|x-y|^{n-\alpha}}|f(y)|\, dy \\
		&\lesssim \sum\limits_{j=M(R)}^\infty (2^j r_b)^\alpha \,\Phi^{-1}\big(\phi(2^j r_b)\big)\|f\|_{L(\Phi,\phi)}
	\end{align*}
	and then
	\begin{align}\label{*1}
		&\nonumber\bigg\|\chi_{B_b}(x)\int_{\R^n} \frac{|\Omega\chi_{(B(R))^c}(x-y)|}{|x-y|^{n-\alpha}}|f(y)|\, dy\bigg\|_{L(\Psi,\phi)}\\
		&\nonumber\lesssim \|\chi_{B_b}\|_{L(\Psi,\phi)}\sum\limits_{j=M(R)}^\infty (2^j r_b)^\alpha \,\Phi^{-1}\big(\phi(2^j r_b)\big)\|f\|_{L(\Phi,\phi)} \\
		&\lesssim \frac{1}{\Phi^{-1}\big(\phi(r_b)\big)}\sum\limits_{j=M(R)}^\infty (2^j r_b)^\alpha\, \Phi^{-1}\big(\phi(2^j r_b)\big)\|f\|_{L(\Phi,\phi)}.
	\end{align}
	Furthermore, for $R$ big enough and any $|x|\leqslant \frac{R}{2}$, $\{y:\  |x-y|>R\}\cap B_b=\varnothing$  and thus
	\begin{align}\label{*2}
		&\nonumber\bigg\|\int_{\R^n} \frac{|\Omega\chi_{(B(R))^c}(x-y)|}{|x-y|^{n-\alpha}}|f\chi_{B_b}(y)|\, dy\bigg\|_{L(\Psi,\phi)}\\
		&\nonumber\leqslant\sum\limits_{j=j_R}^\infty\bigg\|\chi_{B{(2^{j+1}r_b)}\backslash B{(2^{j}r_b)}}(x)\int_{\R^n} \frac{|\Omega\chi_{(B(R))^c}(x-y)|}{|x-y|^{n-\alpha}}|f\chi_{B_b}(y)|\, dy\bigg\|_{L(\Psi,\phi)}\\
		&\nonumber \lesssim \sum\limits_{j=j_R}^\infty (2^j r_b)^{\alpha-\frac{n}{s}}\|f\chi_{B_b}\|_{L^s}\|\chi_{B(2^{j+1}r_b)}\|_{L(\Psi,\phi)}\\
	    &	\lesssim \sum\limits_{j=j_R}^\infty (2^j r_b)^{\alpha-\frac{n}{s}}\|f\|_{L(\Phi,\phi)}\frac{1}{\Psi^{-1}\big(\phi(2^{j+1}r_b)\big)},
	\end{align}
	where $2^{j_R}\leqslant\frac{R}{2}<2^{j_R+1}$ and $j_R\to \infty\ \ as\ \ R\to \infty$.
	
	Combining (\ref{C2}), (\ref{C3}),
	(\ref{assumption}),  (\ref{*0}), (\ref{*1}) and (\ref{*2}), we obtain
	$$\|[b, T^R_\varepsilon]-[b,T_\varepsilon]\|_{L(\Phi,\phi)(\R^n)\to L(\Psi,\phi)(\R^n)}\to 0\ as\ R\to \infty.$$

	Step 3. $$\|[b,\tilde T^1_{\Omega,\alpha}]-[b, T_\varepsilon]\|_{L(\Phi,\phi)(\R^n)\to L(\Psi,\phi)(\R^n)}\to 0\ as\ \varepsilon\to 0.$$\\
	According to Lemma \ref{estimate b}, there exists $0<\theta<1$ such that for any $x,y\in\R^n$ with $|x-y|<1$,
	$$|b(x)-b(y)|\lesssim \|\nabla b\|_{L^\infty}|x-y|^\theta \psi(|x-y|).$$
	Therefore for $\varepsilon<1$,
	\begin{align*}
		&|[b,\tilde T^1_{\Omega,\alpha}]f(x)-[b, T_\varepsilon]f(x)|
		\leqslant \int_{\R^n} |b(x)-b(y)|\frac{|\Omega\chi_{B(\varepsilon)}(x-y)|}{|x-y|^{n-\alpha}}|f(y)|\, dy\\
		&\lesssim \sum\limits_{j=0}^\infty (2^{-j}\varepsilon)^{\theta}
		\psi(2^{-j}\varepsilon)\bigg\|\frac{\Omega(x-\cdot)}{|x-\cdot|^{n-\alpha}}\bigg\|_{L^{s'}\big((B(x,2^{-j-1}\varepsilon))^c\big)}\cdot\|f\|_{L^s\big(B(x,2^{-j}\varepsilon)\big)}\\
		&\lesssim \sum\limits_{j=0}^\infty (2^{-j}\varepsilon)^{\theta+\alpha}\psi(2^{-j}\varepsilon)\bigg(\fint_{B(x,2^{-j}\varepsilon)}|f|^s\bigg)^{\frac{1}{s}}\\
		&\lesssim \varepsilon^\theta \bigg(M_{((\cdot)^\alpha\psi(\cdot))^s}|f|^s\bigg)^{\frac{1}{s}}(x).
	\end{align*}
	From (\ref{C2}), $R^\alpha\psi(R)\Phi^{-1}\big(\phi(R)\big)\lesssim \Psi^{-1}\big(\phi(R)\big)$ for any $R>0$ and thus
	$$\big(R^\alpha\psi(R)\big)^s\{\Phi\big(\cdot^\frac{1}{s}\big)\}^{-1}(\phi(R))\lesssim \{\Psi\big(\cdot^\frac{1}{s}\big)\}^{-1}\big(\phi(R)\big)\ \ for\ \ any\ \ R>0.$$
	This together with Lemma \ref{M Bds} yields
	\begin{align*}
		\bigg\|\bigg(M_{\big((\cdot)^\alpha\psi(\cdot)\big)^s}|f|^s\bigg)^{\frac{1}{s}}\bigg\|_{L(\Psi,\phi)}
		&=\bigg\|M_{\big((\cdot)^\alpha\psi(\cdot)\big)^s}|f|^s\bigg\|^{\frac{1}{s}}_{L(\Psi(\cdot^{\frac{1}{s}}),\phi)}\\
		&\lesssim \||f|^s\|^{\frac{1}{s}}_{L(\Phi(\cdot^{\frac{1}{s}}),\phi)} =\|f\|_{L(\Phi,\phi)}
	\end{align*}
	and consequently
	\begin{align*}
		&\|[b,\tilde T^1_{\Omega,\alpha}]-[b, T_\varepsilon]\|_{L(\Phi,\phi)(\R^n)\to L(\Psi,\phi)(\R^n)}\\
		&=\sup\limits_{\|f\|_{L(\Phi,\phi)}\leqslant 1} \|[b,\tilde T^1_{\Omega,\alpha}]f-[b, T_\varepsilon]f\|_{L(\Psi,\phi)}\\
		&\lesssim \sup\limits_{\|f\|_{L(\Phi,\phi)}\leqslant 1}
		\varepsilon^\theta \bigg\|\bigg(M_{\big((\cdot)^\alpha\psi(\cdot)\big)^s}|f|^s\bigg)^{\frac{1}{s}}\bigg\|_{L(\Psi,\phi)}\\
		&\lesssim \varepsilon^\theta\to 0\ \ as \ \ \varepsilon\to0.
	\end{align*}
	Now the proof of Theorem \ref{compact 2} is completed.
\end{proof}


\end{document}